\documentclass[12pt,reqno]{amsart}

\usepackage[T1]{fontenc}
\usepackage{mathptmx}
\usepackage{microtype}
\usepackage{microtype}

\usepackage[
left=1in,
right=1in,
top=1in,
bottom=1in
]{geometry}

\usepackage{amsmath}
\usepackage{amssymb}
\usepackage{amsfonts}
\usepackage{amsthm}
\usepackage{mathtools}

\numberwithin{equation}{section}

\usepackage{graphicx}
\usepackage{tikz}
\usepackage{float}

\usepackage{enumitem}

\usepackage[
colorlinks=true,
linkcolor=blue,
citecolor=blue,
urlcolor=blue
]{hyperref}

\usepackage{empheq}
\usepackage{ragged2e}

\theoremstyle{plain}

\newtheorem{Theorem}{Theorem}[section]

\newtheorem{Proposition}[Theorem]{Proposition}
\newtheorem{Corollary}[Theorem]{Corollary}

\theoremstyle{Definition}

\newtheorem{Definition}[Theorem]{Definition}
\newtheorem{Assumption}[Theorem]{Assumption}

\theoremstyle{Remark}

\newtheorem{Remark}[Theorem]{Remark}

\allowdisplaybreaks

\AtBeginDocument{
\setlength{\abovedisplayskip}{10pt plus 2pt minus 6pt}
\setlength{\belowdisplayskip}{10pt plus 2pt minus 6pt}
\setlength{\abovedisplayshortskip}{5pt plus 2pt}
\setlength{\belowdisplayshortskip}{6pt plus 2pt minus 6pt}
}

\title[Traveling-Wave Solutions for Geological Transport]
{\textbf{
Traveling-Wave Solutions for an Einstein-Type\\
Material-Balance Model of the Chemotactic Transport}}

\author{Isanka Garli Hevage}
\author{Akif Ibragimov}

\email{igarlihevage@centralstate.edu}
\email{ilya1sergey@gmail.com}

 \AtBeginDocument{
 \addtolength\abovedisplayskip{-0.1\baselineskip}
  \addtolength\belowdisplayskip{-0.1\baselineskip}
  \addtolength\abovedisplayshortskip{-0.1\baselineskip}
  \addtolength\belowdisplayshortskip{-0.1\baselineskip}
 }

\begin{document}

\begin{abstract}
We develop a nonlinear continuum transport model
for localized traveling structures in a coupled two-phase medium.
The model is derived from an Einstein-type material-balance formulation
{ combining diffusion with directed transport induced by the relative
gradient of an evolving background phase.}
The resulting system couples diffusion, nonlinear gradient-driven transport,
and {irreversible background depletion}.
{ 
The framework provides a mathematical model for chemotactic transport and a
possible modeling perspective for processes associated with the formation of
oil and gas deposits.
}
{ 
We analyze traveling-wave solutions and identify $\beta=\chi/D=1$ as the
critical threshold for globally positive localized traveling bands under
prescribed asymptotic conditions as the traveling coordinate $s$ tends to
$\pm\infty$. Such bands exist for
$\beta\geq1$, whereas for $0<\beta<1$ the positive branch reaches zero
background concentration at a finite traveling coordinate. For $\beta\geq1$, explicit profiles are obtained and, for fixed wave speed
$c>0$ and right-end state $V_+>0$, are unique up to spatial translation
in the traveling coordinate $s$.
}
We further derive the
{ variable-coefficient linearized perturbation operator and
establish a conditional finite-time a priori estimate for the
mobile-phase perturbation under bounded background forcing using a
maximum-principle argument.}
\end{abstract}
\maketitle


\section{Introduction}
\label{sec:introduction}

The formation and evolution of localized structures in heterogeneous transport
systems result from the interaction of diffusion, directed transport, and
material transformation processes occurring across multiple spatial and
temporal scales \cite{Barenblatt, Dietrich2018}. Understanding these mechanisms requires mathematical models
capable of describing the coupling between microscopic transport processes and
macroscopic pattern formation in evolving media. Such models arise in a wide range of applications, including biological
systems, porous-media transport, filtration, contaminant migration, and
reactive-flow systems. This type of process can be both organic and inorganic in nature; related
discussions can be found in \cite{ChemotaxisScure}. We model it as a competition between diffusion and drift induced by the
gradient of the ``passive'' matter (background phase). In this work, we develop a
continuum transport framework for the emergence of localized traveling
structures by extending an Einstein-type probabilistic description of
Brownian motion \cite{Einstein05}, to a coupled two-phase system incorporating
nonlinear transport and { background depletion}.

The classical Einstein theory of Brownian motion establishes a fundamental
connection between microscopic particle displacement and macroscopic
diffusion. In the absence of a preferred transport direction, the mean
displacement vanishes, and the second moment of the displacement distribution
determines the continuum diffusive evolution. Although this framework provides
a rigorous foundation for random transport, geological systems exhibit
additional mechanisms arising from spatial heterogeneity \cite{Dietrich2018},
concentration gradients, and interactions between transported materials and the surrounding
medium. These effects generate directional transport mechanisms that cannot be fully
captured by a purely diffusive description. To the best of our knowledge, the
incorporation of this type of material balance structure into mathematical
models of transport and diffusion processes in geological formations has not
been explored. 

{ 
The proposed framework provides a mathematical perspective for analyzing
transport mechanisms associated with differential hydrocarbon entrapment
and the formation of oil and gas deposits, as described by the Gussow
geological model \cite{Gussow2}. An interpretation of the Gussow model
in the context of the present mathematical results is discussed below.
}

To incorporate such mechanisms, we consider a two-phase continuum system.
The first component, denoted by $v(x,t)$, represents a slowly evolving
background phase describing the geological medium. { 
The second component, denoted by $u(x,t)$, represents an active transported
phase whose motion is influenced by spatial variations of the background
phase $v(x,t)$.
}

{ 
For clarity, $u(x,t)$ may be interpreted as a mobile liquid phase and
$v(x,t)$ as a chemically active background medium through which the liquid
is transported.
} The essential modeling assumption is
that the transport of the active phase is not prescribed by an external
velocity field; instead, the transport direction emerges dynamically from the
interaction between the two phases, namely, the liquid and the porous medium.
More precisely, the active component $u(x,t)$ responds to the relative
gradient of the background concentration $v(x,t)$. Regions where the
background phase varies generate a preferred transport direction for $u$. At
the same time, the active phase modifies the background medium through a
depletion mechanism, causing the transport field itself to evolve. Thus, the
model contains a nonlinear feedback mechanism: the background phase determines
the motion of the active phase, while the active phase alters the environment
that controls its transport. This mechanism leads to a coupled transport-reaction system in which
diffusion, directed transport, and  {background depletion interact}. Unlike
classical reaction-diffusion models, where the transport operator is fixed
independently of the evolving solution, the present formulation generates the
transport mechanism through the solution itself. Consequently, the model
provides a nonlinear continuum description of how persistent localized
structures may emerge in heterogeneous media. The proposed framework is related to several nonlinear transport models,
including filtration systems, reactive transport equations,
diffusion--advection models, and gradient-driven aggregation phenomena.
{  Its distinguishing feature is the two-phase Einstein-type material-balance
formulation underlying the continuum system. The displacement balance
determines the evolution equation in terms of the first displacement moment
$E_u$, thereby preserving the connection between microscopic displacement
statistics and macroscopic transport. The dependence of $E_u$ on the
background phase is then specified separately through the constitutive
closure introduced in Section~\ref{par2}.
}

In this regard, the generalized Einstein model can be used to describe
chemotactic interactions in which bacteria move toward nutrient sources and
form traveling bands. A similar approach was employed by Barenblatt
\cite{Barenblatt} to analyze combustion processes in reactive media governed
by KPP-type nonlinear equations originally introduced by Fisher,
Kolmogorov, Piskunov, and Petrovski to model gene propagation phenomena
(see the review in the book \cite{NonLinearPhenomenaMaterialsII}).
The main objective of this work is to determine whether the resulting
nonlinear system supports localized traveling structures. Such structures
represent coherent patterns that propagate through the medium with a constant
velocity while maintaining their spatial profile.

Each modeling framework is associated with specific boundary conditions and
a corresponding class of mathematical problems. One important potential application is the
formation of geological structures. We believe that this type of modeling is
applicable to the interpretation of geological phenomena related to the
formation of oil and gas reserves in heterogeneous porous media (see
\cite{rubin1, Gussow1, Gussow2, NonLinearPhenomenaMaterialsII} and references
therein). To establish the applicability of this type of modeling, it is
important to investigate whether the system of equations obtained from the
Einstein paradigm admits solution patterns similar to the observed patterns of
geological structures. To analyze this question,
we introduce the traveling coordinate
\[
s=x-ct,
\]
where $c>0$ denotes the propagation speed. Under this transformation, the
original partial differential system reduces to an autonomous system of
ordinary differential equations describing the internal structure of the
traveling profile. The traveling-wave formulation provides a natural framework for analyzing
layer formation. A localized traveling structure corresponds to a situation in
which the active phase forms a concentrated band, while the background phase
undergoes a monotone transition.  { 
The persistence of such a structure requires a balance among diffusive
spreading, gradient-induced transport, and background depletion. As shown in
Section~\ref{par3}, the resulting traveling-band structure is governed by the
ratio $\beta=\chi/D$.
}

From the modeling perspective, such solutions describe localized structures
developing within heterogeneous porous formations. The active phase $u(x,t)$
may represent a migrating fluid, chemical component, or transported material,
while the background phase $v(x,t)$ represents the evolving geological
environment. The traveling-wave solution then corresponds to a coherent
layer-like structure generated through nonlinear coupling rather than a
prescribed external flow field. A representative interpretation arises in
subsurface fluid transport through porous media. The migrating component
$u(x,t)$ forms a localized concentration profile, while the background
$v(x,t)$ represents the porous structure. 
{ 
The interaction between transported fluids and geological media has long been
recognized; Lyell described how the solvent action of carbonic-acid-bearing
water can alter rock formations \cite{Lyell1830}.
}
{ Feedback between
transport and an evolving porous medium is well established in reactive-flow
modeling, where fluid--solid interactions can modify medium properties and
thereby alter subsequent transport \cite{LaddSzymczak2021}.} In the present
model, the active phase modifies the background through the prescribed
depletion mechanism, while the evolving background, in turn, regulates the
motion of the active phase. This provides a mathematical mechanism capable of
generating localized traveling structures. The continuum equations developed
in this work are derived from an Einstein-type material-balance principle. In
the classical Brownian setting, symmetric microscopic displacements produce
zero mean drift and lead to the diffusion equation \cite{Einstein05}. In
contrast, the present model incorporates an asymmetric displacement mechanism
induced by the background phase, which generates an additional transport
contribution associated with spatial variations of $v$.

The resulting system combines three mechanisms: diffusion of the active phase,
background-induced directed transport, and nonlinear depletion of the
background. The interplay of these effects determines whether localized
traveling bands can form and persist. The main analytical objective is to construct and characterize localized
traveling bands consisting of a single-hump profile for $u$ coupled to a
monotone transition in $v$.

{ 
The governing dimensionless parameter is
\begin{equation}
    \beta=\frac{\chi}{D}.
\end{equation}
The analysis identifies $\beta=1$ as the critical threshold for the class of
globally positive localized traveling bands considered in this work. Such
profiles exist for $\beta\geq1$ under the prescribed whole-line asymptotic
conditions. The supercritical case $\beta>1$ yields an explicit
algebraic-power profile, whereas the critical case $\beta=1$ admits a
distinct profile. For $0<\beta<1$, the positive
traveling branch reaches $V=0$ at a finite traveling coordinate and therefore
does not satisfy the same globally positive whole-line conditions.
}

{ 
The traveling-wave problem admits a family of profiles parameterized by
$c>0$; when the total mobile mass and asymptotic background states are
prescribed, the speed is fixed by the corresponding mass-balance relation.
To illustrate the traveling-band structure analyzed below,
Figure~\ref{fig:hump} shows the explicit traveling-wave profiles
\[
u(x,t)=U(x-ct),
\qquad
v(x,t)=V(x-ct),
\]
at two distinct times $t_1<t_2$ in the physical coordinate $x$.
Although $U(s)$ and $V(s)$ are stationary in the traveling coordinate
$s=x-ct$, the corresponding profiles in the physical coordinate $x$
translate with speed $c$.
}

{\color{red}

\begin{figure}[H]
\centering
\begin{tikzpicture}[xscale=0.62,yscale=3.0]

\begin{scope}[xshift= -5 cm]

\draw[->,thick] (-4.3,0) -- (5.0,0) node[right] {$x$};
\draw[->,thick] (-4,0) -- (-4,1.15) node[above] {$V$};

\draw[dashed,gray] (-4,1) -- (4.7,1);
\node[left,font=\small] at (-4,1) {$V_+=1$};

\draw[very thick,blue!80!black,
      domain=-4:4.5,samples=150,smooth]
plot (\x,{1/(1+exp(-\x))});

\draw[dashed,thick,blue!50,
      domain=-4:4.5,samples=150,smooth]
plot (\x,{1/(1+exp(-(\x-2)))});

\node[blue!80!black,font=\small] at (-0.9,0.72) {$t_1=0$};
\node[blue!50,font=\small] at (4.5,0.72) {$t_2=2$};

\draw[dashed,gray] (-4,0.5) -- (4.7,0.5);
\node[left,font=\small] at (-4,0.5) {$V_+/2=\frac12$};

\draw[dashed,gray] (0,0) -- (0,0.5);
\draw[dashed,gray] (2,0) -- (2,0.5);

\node[font=\small] at (0,-0.12) {$0$};
\node[font=\small] at (2,-0.12) {$2$};

\node[font=\small] at (0.2,1.12) {Background profile};

\end{scope}

\begin{scope}[xshift= 6 cm]

\draw[->,thick] (-4.3,0) -- (5.0,0) node[right] {$x$};
\draw[->,thick] (-4,0) -- (-4,0.32) node[above] {$U$};

\draw[dashed,gray] (-4,0.25) -- (4.7,0.25);
\node[left,font=\small] at (-4,0.25) {$U^*=\frac14$};

\draw[very thick,brown!80!black,
      domain=-4:4.5,samples=150,smooth]
plot (\x,{exp(\x)/(1+exp(\x))^2});

\draw[dashed,thick,orange!80!red,
      domain=-4:4.5,samples=150,smooth]
plot (\x,{exp(\x-2)/(1+exp(\x-2))^2});

\draw[dashed,gray] (0,0) -- (0,0.25);
\draw[dashed,gray] (2,0) -- (2,0.25);

\node[font=\small] at (0,-0.12) {$0$};
\node[font=\small] at (2,-0.12) {$2$};

\node[brown!80!black,font=\small] at (-1.1,0.35) {$t_1=0$};
\node[orange!80!red,font=\small] at (3.0,0.35) {$t_2=2$};

\node[font=\small] at (0.2,1.12){Mobile profile};

\end{scope}

\end{tikzpicture}
\caption{{Traveling-wave profiles in the physical coordinate $x$ at
$t_1=0$ (solid curves) and $t_2=2$ (dashed curves), evaluated from the
explicit analytical solution with
$D=1$, $\chi=2$, $\beta=\chi/D=2$, $k=1$, $c=1$,
$V_+=1$, and $s_0=0$.
For these parameters, $ \displaystyle V(x,t)=[1+\exp(-(x-t))]^{-1},
\
U(x,t)=\exp(-(x-t))[1+\exp(-(x-t))]^{-2}.
$ 
The profiles at $t_2$ are translations of those at $t_1$ by
$c(t_2-t_1)=2$, with unchanged shape and amplitude. The curves are
obtained by direct evaluation of the explicit analytical formulas and
therefore provide a graphical illustration of the predicted translation
of the traveling profiles.}} 
\label{fig:hump}
\end{figure}
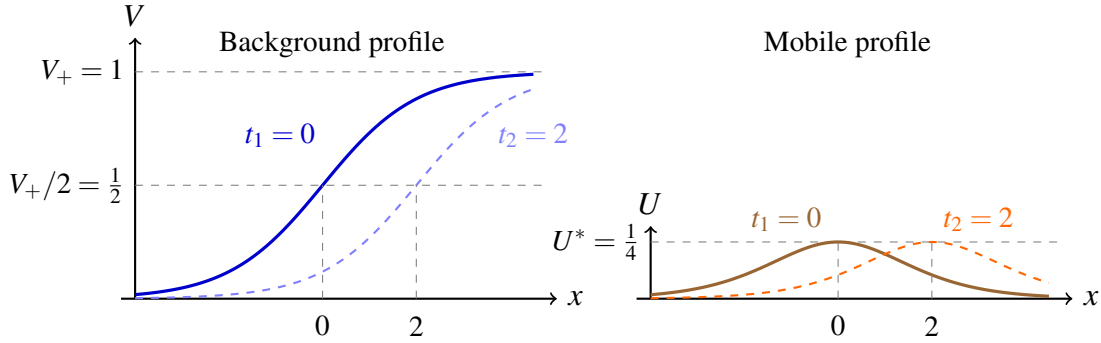
}

{ 
Traveling-wave solutions in chemotaxis systems, including models with
logarithmic or singular sensitivity, have been studied extensively
analytically and numerically; see \cite{WangTravelingWavesReview} and the
references therein. In \cite{WangTravelingWavesReview}, traveling-wave
solutions under different asymptotic end-state conditions are considered for
Keller--Segel-type chemotaxis systems with logarithmic sensitivity, including
pulse-type profiles that provide a point of comparison with the traveling-band
structure considered in Definition~\ref{strong-hump-traveling-band}.
More generally, existence, wave-speed selection, asymptotic behavior, and
stability have been analyzed for several classes of chemotaxis models.
Asymptotic stability has also been established for a related chemotaxis
PDE--ODE system with nonlinear diffusion in \cite{GhaniLiZhang2021}.
Complementary spatial-dynamical approaches have been developed for related
chemotaxis traveling-wave problems. In particular, Qiao \cite{Qiao2024}
employed geometric singular perturbation theory together with invariant-manifold
methods to establish traveling-wave existence in a volume-filling chemotaxis
model and subsequently analyzed the spectral instability of the resulting
waves. In contrast, the traveling-band profiles considered here are obtained
directly from the explicit traveling-wave reduction of the two-phase
Einstein-type material-balance system.

The novelty of the present work therefore does not lie in the general
existence of traveling waves in chemotaxis. Rather, it concerns the
traveling-band structure generated by the specific two-phase Einstein-type
material-balance formulation developed here, together with its explicit
characterization and a conditional finite-time a priori estimate for the
mobile-phase perturbation under bounded background forcing.
}

The main contributions of this work are summarized as follows. First, we
{ 
derive a coupled two-phase continuum system from an Einstein-type
material-balance formulation in which the mobile-phase transport is coupled
to the evolution and depletion of a background phase.
}

{ 
Second, we identify $\beta=1$ as the critical threshold for globally positive
localized traveling bands under the prescribed whole-line asymptotic
conditions. We establish existence for $\beta\geq1$, derive explicit profiles
for both $\beta>1$ and $\beta=1$, and show that the subcritical branch
$0<\beta<1$ reaches $V=0$ at a finite traveling coordinate.
}

Third, { for fixed $c>0$ and $V_+>0$,} we establish uniqueness
of the resulting traveling-band profile up to spatial translation.

Fourth, we derive the variable-coefficient linearized
perturbation system about the traveling band and { establish a
conditional finite-time a priori estimate for the mobile-phase perturbation
under bounded background forcing} by a maximum-principle argument.

Finally, we reduce the
traveling-wave system to a nonlinear third-order scalar equation for the
background profile, providing an alternative characterization of the coherent
structure.

The explicit traveling-wave construction establishes a direct connection
between microscopic displacement assumptions and macroscopic pattern
formation.  In particular, the analysis demonstrates how a probabilistic
transport mechanism can generate persistent nonlinear structures through
{ the coupling of mobile-phase transport and background depletion}.
This provides a mathematical framework for studying the formation of localized transport structures in
heterogeneous media and { their associated linearized perturbation
dynamics.} The present work focuses on the existence, explicit characterization, and
structural properties of traveling bands. The stability and long-time dynamics
of these structures require a separate analysis of the linearized and
nonlinear evolution. The analysis developed here provides the perturbative
framework required for future spectral and nonlinear stability investigations.

The proposed model should be interpreted as a mathematical framework for
studying transport-induced pattern formation in heterogeneous media. The
model parameters and interaction mechanisms may ultimately be calibrated with
respect to specific physical systems. Nevertheless, the analytical results
identify the parameter regime in which localized traveling structures can
form and characterize their underlying profile properties.
Because the linearized problem is formulated on the whole line
$
\mathbb{R}=(-\infty,\infty),
$
the perturbation operator contains variable coefficients generated by the
traveling profile. { After establishing their boundedness and
far-field behavior at $\pm \infty$, we obtain a conditional finite-time a priori estimate
for the mobile-phase perturbation on bounded spatial intervals by a
maximum-principle argument in the spirit of Landis \cite{Landis}.}

The remainder of the paper is organized as follows. In
Section~\ref{par2}, we derive an Einstein-type material-balance formulation
for the two-phase system and obtain the corresponding continuum
diffusion--drift--depletion equations. This establishes the connection
between the microscopic displacement mechanism and the resulting nonlinear
transport model.

In Section~\ref{par3}, we analyze the associated traveling-wave problem. We
prove the existence of a strong one-hump traveling band, derive explicit
traveling profiles, determine the location and amplitude of the maximum, and establish uniqueness up to spatial translation
{ for fixed $c>0$ and $V_+>0$.}

In Section~\ref{par4}, we linearize the system around the traveling-band
solution and characterize the resulting variable-coefficient perturbation
operator. We determine the asymptotic behavior of the linearized
coefficients and { establish a conditional finite-time
maximum-principle estimate for the mobile-phase perturbation under bounded
background forcing.}

Finally, in Section~\ref{par5}, we derive a reduced scalar characterization
of the traveling-wave profile by eliminating the mobile component. This
formulation provides an alternative representation of the coherent structure
and clarifies its dependence on the model parameters.

\section{Einstein-Type Material Balance Formulation for Coupled Phase Transport}
\label{par2}
We formulate a continuum transport model based on an Einstein-type
material-balance framework. The classical theory of Brownian motion establishes
a fundamental connection between microscopic displacement statistics and
macroscopic diffusion. We extend this framework to a coupled two-phase medium
in which the transport of an active phase is influenced by the evolution of a
background phase.  Let $u(x,t)$ denote the concentration of the mobile phase and let $v(x,t)$
denote the background phase. The microscopic transport of the mobile phase is
described by a displacement probability density function
$ \displaystyle
\varphi_{u,v}(x,t,\xi_u),
$
where $\xi_u$ denotes the microscopic displacement from the point $x$ at time
$t$. The dependence of the probability density on both phases reflects the
assumption that the local transport mechanism is modified by the state of the
surrounding medium. The first two moments of the mobile-phase displacement distribution are defined
by
\begin{equation}
E_u(x,t)
=
\int_{\mathbb R}
\xi_u\,\varphi_{u,v}(x,t,\xi_u)\,d\xi_u ,
\label{eq:Eu}
\end{equation}
and
\begin{equation}
\sigma_u^2(x,t)
=
\int_{\mathbb R}
\xi_u^2\,\varphi_{u,v}(x,t,\xi_u)\,d\xi_u .
\label{eq:sigmau}
\end{equation}
The first moment represents the mean microscopic displacement, whereas the
second moment determines the contribution of microscopic fluctuations to
macroscopic diffusion.

For the background phase, we assume that microscopic motion is negligible on
the transport time scale under consideration. Accordingly, we set
{ 
\begin{equation}
\varphi_{v,u}(x,t,\xi_v)=\delta(\xi_v),
\label{eq:delta}
\end{equation}
where $\delta$ denotes the Dirac measure concentrated at $\xi_v=0$,
}
which implies
\begin{equation}
E_v(x,t)=0,
\qquad
\sigma_v^2(x,t)=0 .
\label{eq:backgroundmoments}
\end{equation}

Thus, the background component does not generate an independent diffusive
flux. Instead, it evolves through interaction with the mobile phase and
modifies the transport environment experienced by $u$. { 
More generally, the displacement kernels are assumed to be nonnegative and
normalized:
\begin{equation}
\varphi_{u,v}(x,t,\xi)\geq0,
\qquad
\varphi_{v,u}(x,t,\xi)\geq0,
\qquad
\int_{\mathbb R}\varphi_{u,v}(x,t,\xi)\,d\xi
=
\int_{\mathbb R}\varphi_{v,u}(x,t,\xi)\,d\xi
=1.
\label{eq:kernel-normalization}
\end{equation}
Consequently, for any $f(x,t)$ independent of the displacement variable,
\[
\int_{\mathbb R}\varphi(x,t,\xi)f(x,t)\,d\xi=f(x,t),
\]
where $\varphi$ denotes either $\varphi_{u,v}$ or $\varphi_{v,u}$.} We incorporate this interaction through a local depletion mechanism with rate
coefficient $k>0$. The corresponding material-balance relations are
\begin{align}
u(x,t+\tau_u)
&=
\int_{\mathbb R}
u(x+\xi_u,t)
\varphi_{u,v}(x,t,\xi_u)\,d\xi_u
+
u\,\partial_xE_u ,
\label{material-u}
\\
v(x,t+\tau_v)
&=
\int_{\mathbb R}
v(x+\xi_v,t)
\varphi_{v,u}(x,t,\xi_v)\,d\xi_v
-ku .
\label{material-v}
\end{align}
The first relation accounts for the contribution of microscopic displacement
and the spatial variation of the mean displacement field.
 {The second relation describes the irreversible depletion of the background phase at a
rate determined by the mobile-phase concentration.}

{ 
The introduction of the evolving background phase $v$ extends the classical
single-phase Einstein material-balance formulation to a coupled two-phase
system. The first displacement moment $E_u$ retains its role as the mean
directed displacement of the mobile phase, while the additional balance law
couples this displacement mechanism to the evolution of the background phase.
This two-phase material-balance structure is a principal modeling distinction
of the present formulation from the classical Einstein diffusion framework.
}

To obtain the continuum limit, the displacement terms are expanded for
sufficiently regular concentration fields.
The mobile-phase contribution
satisfies
\begin{equation}
u(x+\xi_u,t)-u(x,t)
=
\xi_u u_x
+\frac12\xi_u^2u_{xx}
+O(\xi_u^3),
\label{eq:taylor-space}
\end{equation}
while the temporal increment satisfies
\begin{equation}
u(x,t+\tau_u)-u(x,t)
=
\tau_u u_t+O(\tau_u^2).
\label{eq:taylor-time}
\end{equation}
Substituting these expansions into the material-balance relation and retaining
the leading-order moment contributions gives
\begin{equation}
\tau_u u_t
=
\partial_x(uE_u)
+
\frac{\sigma_u^2}{2}u_{xx}.
\label{eq:continuum-u}
\end{equation}
For the background phase, the Dirac measure eliminates the diffusive
contribution and yields
\begin{equation}
\tau_v v_t=-ku .
\label{eq:continuum-v}
\end{equation}
After normalization of the characteristic time scales,
\begin{equation}
\tau_u=\tau_v=1,
\label{eq:timescale}
\end{equation}
and assuming a constant microscopic variance,
\begin{equation}
\frac{\sigma_u^2}{2}=D,
\label{eq:variance}
\end{equation}
we obtain the intermediate continuum system
\begin{align}
u_t
&=
\partial_x(uE_u)+Du_{xx},
\label{intermediate-u}
\\
v_t
&=
-ku .
\label{intermediate-v}
\end{align}

{ 
Although the Taylor expansion and moment truncation used above are
mathematically related to Kramers--Moyal-type continuum-limit arguments
\cite{Kramers1940,Moyal1949}, the present derivation is formulated directly
from the Einstein-type material-balance relations
\eqref{material-u}--\eqref{material-v}, rather than from a Kramers--Moyal
expansion}. {Thus, the material-balance construction yields a coupled continuum system in
which the directed transport of $u$ is represented by its first displacement
moment $E_u$, while the background phase evolves according to the interaction law
$v_t=-ku$. Here, $E_u$ represents the mean directed displacement of the
mobile phase induced by its interaction with the background phase $v$.
The Einstein-type material-balance argument determines the evolution equation
in terms of $E_u$ but does not determine its dependence on the background
field. The model is closed by the constitutive relation
\begin{equation}
E_u
=
-\chi\frac{v_x}{v}
=
-\chi\,\partial_x
\log\left(\frac{v}{v_{\rm ref}}\right),
\qquad v>0, 
\label{eq:constitutive} 
\end{equation} 
where $\chi>0$ is the coefficient of background-induced transport. Here $v_{\rm ref}>0$ is a fixed reference concentration introduced so that the argument of the logarithm is dimensionless. Since $v_{\rm ref}$ is spatially constant,
\begin{equation}
\partial_x\log\left(\frac{v}{v_{\rm ref}}\right)
=
\frac{v_x}{v},
\end{equation}
and therefore the reference concentration does not alter the governing
transport equation. Accordingly, \eqref{eq:constitutive} is a constitutive modeling assumption, rather than a consequence of the material-balance expansion. The relative-gradient dependence in \eqref{eq:constitutive} distinguishes this closure from a linear gradient law proportional to $v_x$.}
{ 
Related relative-gradient transport laws also arise in chemotactic and
diffusiophoretic transport in porous media, where the drift velocity may take
the form
$\mathbf{v}_{\mathrm{drift}}=\Gamma\nabla c/c
=\Gamma\nabla\log c$ \cite{SoulaineRoman2026}.
}

{ 
A closure proportional to $v_x$ responds to the absolute spatial variation of
the background field, whereas $v_x/v=\partial_x\log v$ responds to its
relative spatial variation. Thus, the latter describes transport driven by
fractional changes in the background level rather than by the absolute
gradient alone. In particular,
$\displaystyle {\partial_x(av)}/{av}={v_x}/{v}$ for $a>0$.
Thus, the mean directed displacement $E_u$ is determined by the relative
spatial variation of the background phase and is invariant under uniform
rescaling of $v$. Moreover, $v_x/v$ has dimension $\mathrm{L}^{-1}$, and
hence $\chi$ has the same transport dimension as the diffusion coefficient
$D$.

For the symmetric Brownian displacement underlying the classical Einstein diffusion model \cite{Einstein05}, the directed first moment vanishes. Within the present formulation, this limit corresponds to $\chi=0$, yielding $E_u=0$ and reducing \eqref{intermediate-u} to the classical diffusion equation. For $\chi>0$, the logarithmic sensitivity requires $v>0$ and becomes singular as $v\to0^+$; the analysis is therefore confined to regimes in which the background phase remains positive.} Substitution of \eqref{eq:constitutive} into \eqref{intermediate-u}--\eqref{intermediate-v} gives 
\begin{align}
u_t &= -\chi \left( u\frac{v_x}{v} \right)_x + Du_{xx}, \label{model-u} 
\\ v_t &= -ku . \label{model-v} \end{align} 

 {The depletion term in \eqref{model-v} represents an irreversible loss of
the background phase rather than a transfer of material from $v$ to $u$.
Accordingly, $(u,v)$ is not interpreted as a closed mass-conserving
two-phase system. Under no-flux boundary conditions on a bounded domain
$\Omega$, or under sufficient decay at spatial infinity, integration of
\eqref{model-u} gives
\begin{equation}
\frac{d}{dt}\int_{\Omega}u(x,t)\,dx=0,
\label{eq:u-mass-conservation}
\end{equation}
whereas \eqref{model-v} gives
\begin{equation}
\frac{d}{dt}\int_{\Omega}v(x,t)\,dx
=
-k\int_{\Omega}u(x,t)\,dx.
\label{eq:v-mass-depletion}
\end{equation}
Consequently,
\begin{equation}
\frac{d}{dt}\int_{\Omega}\bigl(u(x,t)+v(x,t)\bigr)\,dx
=
-k\int_{\Omega}u(x,t)\,dx.
\label{eq:total-mass-balance}
\end{equation}
Thus, the term $-ku$ describes irreversible depletion of the background
material, with the depleted material not represented as an additional phase
in the present model.
}

{ 
The logarithmic constitutive law further requires strict positivity of the
background phase. From \eqref{model-v}, direct integration in time yields
\begin{equation}
v(x,t)
=
v_0(x)
-
k\int_0^t u(x,\tau)\,d\tau,
\label{eq:v-representation}
\end{equation}
where $v_0(x)=v(x,0)$. Hence, positivity of $v$ is not automatically
preserved for arbitrary initial data and arbitrary times. In particular,
on a time interval $[0,T]$, the condition
\begin{equation}
k\int_0^t u(x,\tau)\,d\tau
<
v_0(x),
\qquad
x\in\Omega,\quad 0\leq t\leq T,
\label{eq:v-positivity-condition}
\end{equation}
ensures
\begin{equation}
v(x,t)>0,
\qquad
x\in\Omega,\quad 0\leq t\leq T.
\end{equation}
Accordingly, the model with logarithmic sensitivity is considered on time
intervals for which \eqref{eq:v-positivity-condition} holds.
}

 {The background evolution law distinguishes the resulting system from concentration-dependent consumption models. In \eqref{model-v}, the depletion rate is determined solely by the mobile phase $u$ and is independent of the local background concentration $v$, whereas the corresponding evolution law in \cite{GhaniLiZhang2021} depends explicitly on both phases. Consequently, $u$ governs depletion of the background phase, while $v$ governs the directed transport of $u$ through the relative gradient $v_x/v$.
Thus, although logarithmic-sensitivity chemotaxis systems and their traveling
waves have been studied previously \cite{WangTravelingWavesReview,GhaniLiZhang2021},
the present system arises from the two-phase Einstein-type material-balance
formulation developed above and has the distinct background evolution law
$v_t=-ku$.
}

{ 
A direct precedent for the relative-gradient structure is provided by
logarithmic sensing in bacterial chemotaxis. Experiments and multiscale
modeling for \emph{Escherichia coli} show that the chemotactic drift responds
to the relative ligand gradient
${\nabla C}/{C}=\nabla\log C$, for $C>0$,
rather than to the absolute concentration gradient alone
\cite{Kalinin2009}. Related coarse-grained descriptions of logarithmic
sensing are given in \cite{Curk2012}. This structure provides a physical
analogue for the logarithmic dependence in \eqref{eq:constitutive}, while
the present model incorporates it through the first displacement moment
$E_u$ of the Einstein-type material-balance formulation.}

 {
A corresponding modeling interpretation may be considered for transport in
heterogeneous geological media. Secondary hydrocarbon migration through
carrier beds occurs along pathways influenced by the geometry and hydraulic
properties of the porous medium, with buoyancy providing an important driving
mechanism in established models \cite{Lehner1988}. The present formulation
does not resolve the full set of forces governing hydrocarbon migration but
isolates a directed-transport mechanism associated with spatial variation of
the background field. Identification of $v$ with a specific measurable
geological quantity and quantitative determination of $\chi$ require
application-specific geological and reservoir data that are beyond the scope
of the present study. Accordingly, the geological interpretation is presented
as a constitutive modeling framework rather than as an empirically calibrated
law for secondary hydrocarbon migration.}

The resulting model combines three mechanisms: diffusion generated by
microscopic random motion, nonlinear gradient-driven transport induced by the
background phase, and depletion of the background through interaction with the
mobile component. Unlike classical drift--diffusion systems with a prescribed
transport field, the drift term here is dynamically generated by the evolving
background variable. This coupled nonlinear system provides the foundation for
the traveling-wave analysis developed in Section~\ref{par3}, where we establish
the existence, explicit structure, and qualitative properties of localized
traveling bands.

{ 
\subsection{Nondimensionalization}

Let $L>0$ denote a characteristic length scale, let $U_0>0$ and $V_0>0$
denote characteristic scales for the mobile and background phases,
respectively, and introduce
\begin{equation}
x=L\hat{x},
\qquad
t=T\hat{t},
\qquad
u=U_0\hat{u},
\qquad
v=V_0\hat{v}.
\label{eq:nondimensional-variables}
\end{equation}
Substitution into \eqref{model-u}--\eqref{model-v} gives
\begin{align}
\hat{u}_{\hat{t}}
&=
-\frac{\chi T}{L^2}
\left(
\hat{u}\frac{\hat{v}_{\hat{x}}}{\hat{v}}
\right)_{\hat{x}}
+
\frac{DT}{L^2}\hat{u}_{\hat{x}\hat{x}},
\label{eq:nondim-u-general}
\\
\hat{v}_{\hat{t}}
&=
-\frac{kU_0T}{V_0}\hat{u}.
\label{eq:nondim-v-general}
\end{align}
Choosing the diffusive time scale
\begin{equation}
T=\frac{L^2}{D},
\end{equation}
reduces the system to
\begin{align}
\hat{u}_{\hat{t}}
&=
-\beta
\left(
\hat{u}\frac{\hat{v}_{\hat{x}}}{\hat{v}}
\right)_{\hat{x}}
+
\hat{u}_{\hat{x}\hat{x}},
\label{eq:nondim-u}
\\
\hat{v}_{\hat{t}}
&=
-\kappa\hat{u},
\label{eq:nondim-v}
\end{align}
where
\begin{equation}
\beta=\frac{\chi}{D},
\qquad
\kappa=\frac{kU_0L^2}{DV_0}.
\label{eq:dimensionless-parameters}
\end{equation}
Thus, $\beta$ is the dimensionless ratio of the background-induced
transport coefficient to the diffusion coefficient and is the parameter
governing the traveling-wave classification developed in
Section~\ref{par3}. The parameter $\kappa$ measures the strength of
irreversible background depletion on the diffusive time scale.
}

\section{Existence and Structure of Traveling Bands}
\label{par3}

We now analyze traveling-wave solutions of the nonlinear transport system
derived in Section~\ref{par2}. These solutions represent coherent structures
whose profiles remain invariant in a frame moving with a constant propagation
speed. Such structures arise from the balance between diffusion,
background-induced transport, and { background depletion}.
We consider solutions on the whole spatial domain
$ \displaystyle
(x,t)\in\mathbb{R}\times[0,\infty),
$
and introduce the traveling coordinate
\begin{equation}
    s=x-ct,
\qquad c>0,
\end{equation}
where $c$ denotes the propagation speed. 

We seek solutions of the form
\begin{align}
    u(x,t)=U(s),
\qquad
v(x,t)=V(s).
\end{align}
By the chain rule,
\begin{equation}
u_t=-cU',
\qquad
v_t=-cV',
\qquad
u_{xx}=U'',
\end{equation}
where the prime denotes differentiation with respect to $s$. Substitution
into the continuum system
\begin{align}
u_t
&=
-\chi
\left(
u\frac{v_x}{v}
\right)_x
+
Du_{xx},\\
v_t &=-ku ,
\end{align}
gives the traveling-wave system
\begin{align}
-cU'
&=
-\chi
\left(
\frac{UV'}{V}
\right)'
+
DU'',
\label{tw-u}
\\
-cV'
&=
-kU .
\label{tw-v}
\end{align}
The traveling-wave transformation converts the coupled partial differential
equations into an autonomous system of ordinary differential equations. The
first equation represents the balance between diffusion and nonlinear
background-driven transport, while the second equation couples the background
evolution to the depletion produced by the mobile phase.

From the second equation,
\begin{equation}
V'=\frac{k}{c}U .
\label{Vprime-basic}
\end{equation}
Therefore, the traveling-wave problem can be written as
\begin{align}
DU''
-
cU'
-
\chi
\left(
\frac{UV'}{V}
\right)'
&=0,
\label{ode-system-one}
\\
V'
-
\frac{k}{c}U
&=0 .
\label{ode-system-two}
\end{align}
This relation plays a central role in the analysis: once the background
profile is known, the mobile component is recovered directly from its
derivative.
\begin{Definition}[Traveling-wave profile]
\label{def-traveling-wave}
A pair $(U,V)$ is called a traveling-wave profile with speed $c>0$ if it satisfies \eqref{ode-system-one}--\eqref{ode-system-two} and the corresponding solution of the original system is given by
\begin{equation}
    u(x,t)=U(x-ct),
\qquad
v(x,t)=V(x-ct).
\end{equation}
\end{Definition}
The objective is to identify localized traveling structures in which the
mobile phase forms a concentrated band while the background phase undergoes a
monotone transition. Accordingly, we introduce the following notion.
\begin{Definition}[One-hump traveling band]
\label{strong-hump-traveling-band}
A traveling-wave profile $(U,V)$ is called a one-hump traveling band if it satisfies the following properties:
\begin{enumerate}
\item[\rm (i)] Both phases remain positive:
\begin{equation}
    U(s)>0,
\qquad
V(s)>0,
\qquad s\in\mathbb{R}.
\end{equation}
\item[\rm (ii)] The profile possesses finite asymptotic states. There exist constants $U_\pm$ and $V_\pm$ such that
\begin{align}
\lim_{s\to-\infty}U(s)&=U_-,
&
\lim_{s\to-\infty}V(s)&=V_-,
\\
\lim_{s\to+\infty}U(s)&=U_+,
&
\lim_{s\to+\infty}V(s)&=V_+ .
\end{align}
\item[\rm (iii)] The mobile phase has a unique maximum. There exists a unique $s^*\in\mathbb{R}$ such that
\begin{equation}
    U(s^*)=\max_{s\in\mathbb{R}}U(s).
\end{equation}
Moreover,
\begin{equation}
U'(s)>0,\qquad s<s^*,
\end{equation}
and
\begin{equation}
U'(s)<0,\qquad s>s^* .
\end{equation}
\item[\rm (iv)] The background phase is monotone:
\begin{equation}
V'(s)>0,
\qquad s\in\mathbb{R}.
\end{equation}
\end{enumerate}
\end{Definition}

Here positivity holds for every finite $s\in\mathbb{R}$, while a profile
component is permitted to converge to zero asymptotically at a spatial
infinity.
The one-hump condition expresses localization of the mobile phase, while the
monotone background profile { describes the corresponding spatial
transition between the depleted and positive background states.}

\begin{Remark}[Translation invariance]
\label{translation-Remark}
Because the traveling-wave system is autonomous, if $(U,V)$ is a profile then so is
\begin{equation}
(U(s+s_0),V(s+s_0))
  \end{equation}
for any constant $s_0\in\mathbb{R}$. { 
Therefore, spatial translation generates an equivalent traveling-wave
profile for fixed $c>0$ and $V_+>0$
}. The asymptotic states in the traveling coordinate correspond to the spatial limits of the original solution:
\[
x\to-\infty \Longleftrightarrow s\to-\infty,
\]
and
\[
x\to+\infty \Longleftrightarrow s\to+\infty .
\]
\end{Remark}

{ 
As established by the nondimensionalization in Section~\ref{par2}, the
relative strength of background-induced transport and diffusion is measured
by the dimensionless parameter
\begin{equation}
\beta=\frac{\chi}{D}>0.
\end{equation}
The qualitative structure of the traveling profile depends on the value of
$\beta$.

We therefore do not impose $\beta>1$ at this stage. Instead, we first derive the relation between the two components of the traveling wave and then distinguish the supercritical case $\beta>1$, the critical case $\beta=1$, and the subcritical case $0<\beta<1$. We impose the following right far-field conditions. 
\begin{Assumption} \label{U-asymp} The traveling profile satisfies \begin{equation} U(+\infty)=0, \qquad 0<V(+\infty)=V_+<\infty . \end{equation} \end{Assumption} Thus, the mobile phase vanishes in the right far field, whereas the background phase approaches the finite positive state $V_+$.
}

From the traveling-wave relation
\[
V'=\frac{k}{c}U ,
\]
the asymptotic condition in Assumption~\ref{U-asymp} implies
\begin{equation}
V'(s)\rightarrow0,
\qquad s\rightarrow+\infty .
\end{equation}
We also require the compatibility condition
\begin{equation}
U'(s)\rightarrow0,
\qquad s\rightarrow+\infty .
\label{Uprime-asymp}
\end{equation}
The following theorem gives the relation between the two components of the traveling profile for all $\beta>0$.

\begin{Theorem}[Relation between the traveling-wave components] \label{trav-band-case1} Let $\beta>0$, $k>0$, and let $(U,V)$ be a traveling-wave solution of \eqref{ode-system-one}--\eqref{ode-system-two}. Suppose that \[ V(s)>0 \] and that Assumption~\ref{U-asymp} and \eqref{Uprime-asymp} hold. If $\beta\neq1$, then 
\begin{equation} U = \frac{c^2}{kD(\beta-1)} V \left[ 1- \left( \frac{V}{V_+} \right)^{\beta-1} \right]. \label{u-v-relate} \end{equation} 
 {If $\beta=1$, then \begin{equation} U = \frac{c^2}{kD} V\log\left(\frac{V_+}{V}\right). \label{u-v-relate-critical} \end{equation} }
\end{Theorem} 
\begin{proof} Integrating the first traveling-wave equation \[ -cU' = -\chi \left( \frac{UV'}{V} \right)' + DU'' \] from $s$ to $+\infty$ gives \[ cU = \chi\frac{UV'}{V} - DU', \] where \[ U(+\infty)=0, \qquad U'(+\infty)=0, \qquad V'(+\infty)=0 \] have been used. Hence, \[ U' = \left( \frac{\chi V'}{DV} - \frac{c}{D} \right)U = \left( \beta\frac{V'}{V} - \frac{c}{D} \right)U. \]

Using
\[
V'=\frac{k}{c}U,
\]
the identity
\[
U'
=
\left(
\beta\frac{V'}{V}
-
\frac{c}{D}
\right)U
\]
can be written as
\[
U'
=
\beta\frac{V'}{V}U
-
\frac{c^2}{kD}V'.
\]
Multiplying by $V^{-\beta}$ gives
\[
V^{-\beta}U'
-
\beta V^{-\beta-1}V'U
=
-\frac{c^2}{kD}V^{-\beta}V'.
\]
Hence,
\begin{equation}
\frac{d}{ds}\left(V^{-\beta}U\right)
=
-\frac{c^2}{kD}V^{-\beta}V'.
\label{UV-exact-derivative}
\end{equation}

For $\beta\neq1$, since
\[
V^{-\beta}V'
=
\frac{1}{1-\beta}
\frac{d}{ds}\left(V^{1-\beta}\right),
\]
integration of \eqref{UV-exact-derivative} gives
\[
V^{-\beta}U
=
\frac{c^2}{kD(\beta-1)}V^{1-\beta}
+
C_0.
\]
Therefore,
\[
U
=
\frac{c^2}{kD(\beta-1)}V
+
C_0V^\beta .
\]

The far-field condition $U(+\infty)=0$ together with
$V(+\infty)=V_+$ gives
\[ C_0 = -\frac{c^2}{kD(\beta-1)} V_+^{1-\beta}, \] and therefore \[ U = \frac{c^2}{kD(\beta-1)} V \left[ 1- \left( \frac{V}{V_+} \right)^{\beta-1} \right]. \]

 { 
For $\beta=1$, \eqref{UV-exact-derivative} becomes
\[
\frac{d}{ds}\left(\frac{U}{V}\right)
=
-\frac{c^2}{kD}\frac{V'}{V}
=
-\frac{c^2}{kD}\frac{d}{ds}\log V.
\]
Hence
\[
\frac{U}{V}
=
C_1-\frac{c^2}{kD}\log V,
\]
and therefore
\[
U
=
V\left(
C_1-\frac{c^2}{kD}\log V
\right).
\]
The far-field conditions
\[
U(+\infty)=0,
\qquad
V(+\infty)=V_+
\]
give \[ C_1=\frac{c^2}{kD}\log V_+, \] and consequently \[ U = \frac{c^2}{kD} V\log\left(\frac{V_+}{V}\right). \] 
}
\end{proof}

 { 
\begin{Proposition}[Critical and subcritical regimes] \label{prop:critical-subcritical} Consider positive traveling-wave profiles satisfying
 \begin{equation} V(-\infty)=0, \qquad V(+\infty)=V_+>0, \qquad U(-\infty)=U(+\infty)=0. \label{eq:whole-line-band-asymptotics} \end{equation}
If $\beta=1$, then, { for fixed $c>0$ and $V_+>0$,}
the system admits a positive localized traveling band on $\mathbb{R}$.
Up to translation, the profile is
 \begin{align} V(s) &= V_+ \exp\left[ -\exp\left( -\frac{c}{D}(s-s_0) \right) \right], \label{critical-V} \\ U(s) &= \frac{c^2V_+}{kD} \exp\left( -\frac{c}{D}(s-s_0) \right) \exp\left[ -\exp\left( -\frac{c}{D}(s-s_0) \right) \right]. \label{critical-U} \end{align} 
 The mobile component has a unique global maximum. If $0<\beta<1$, no strictly positive traveling-wave profile on $\mathbb{R}$ satisfies \eqref{eq:whole-line-band-asymptotics}. The corresponding positive branch reaches $V=0$ at a finite value of the traveling coordinate. \end{Proposition}
 \begin{proof} For $\beta=1$, Theorem~\ref{trav-band-case1} gives \[ U = \frac{c^2}{kD} V\log\left(\frac{V_+}{V}\right). \] Using $V'=(k/c)U$, we obtain \[ V' = \frac{c}{D} V\log\left(\frac{V_+}{V}\right). \] Set \[ Y=\log\left(\frac{V_+}{V}\right). \] Then \[ Y' = -\frac{V'}{V} = -\frac{c}{D}Y, \] and therefore, after incorporating the positive integration constant into the translation parameter $s_0$, \[ Y(s) = \exp\left( -\frac{c}{D}(s-s_0) \right). \] This gives \eqref{critical-V}, while \eqref{critical-U} follows from $U=(c/k)V'$. Directly, \[ V(-\infty)=0, \qquad V(+\infty)=V_+, \qquad U(-\infty)=U(+\infty)=0. \] Moreover, \[ U = \frac{c^2V_+}{kD}Ye^{-Y}, \] and $Ye^{-Y}$ has a unique maximum at $Y=1$. Thus the critical profile is a one-hump localized traveling band. Now let $0<\beta<1$. Equation~\eqref{u-v-relate} can be written as \[ U = \frac{c^2}{kD(1-\beta)} V \left[ \left( \frac{V_+}{V} \right)^{1-\beta} -1 \right], \qquad 0<V<V_+. \] Consequently, \[ V' = \frac{c}{D(1-\beta)} V \left[ \left( \frac{V_+}{V} \right)^{1-\beta} -1 \right]. \] As $V\rightarrow0^+$, \[ V' \sim \frac{cV_+^{1-\beta}}{D(1-\beta)} V^\beta. \] Hence \[ \frac{ds}{dV} \sim \frac{D(1-\beta)}{cV_+^{1-\beta}} V^{-\beta}. \] Since \[ \int_0^\varepsilon V^{-\beta}\,dV<\infty \qquad \text{for }0<\beta<1, \] the positive branch reaches $V=0$ after a finite change in $s$. It therefore cannot remain strictly positive on $\mathbb{R}$ while approaching $V=0$ only as $s\rightarrow-\infty$. \end{proof}
 The preceding result gives the sharp parameter classification for the whole-line positive traveling bands considered here. Under the asymptotic conditions \[ V(-\infty)=0, \qquad V(+\infty)=V_+>0, \qquad U(-\infty)=U(+\infty)=0, \] a global positive localized traveling band exists for $\beta\geq1$. The case $\beta>1$ yields the algebraic-power family analyzed below, whereas $\beta=1$ gives the critical profile \eqref{critical-V}--\eqref{critical-U}. For $0<\beta<1$, the positive branch reaches $V=0$ at a finite traveling coordinate and therefore does not belong to the whole-line positive traveling-wave class considered here. We now restrict the remaining explicit analysis to the supercritical regime $\beta>1$. }

For the remainder of this subsection, assume $\beta>1$. In this
supercritical regime, relation~\eqref{u-v-relate} determines the qualitative
structure of the traveling band. The relation obtained in Theorem~\ref{trav-band-case1} reduces the traveling-wave
system to a scalar first-order equation for the background profile. This
equation determines both the existence and the explicit form of the coherent
structure.
\begin{Proposition}[Explicit representation in the supercritical regime]
\label{explicit-profile}

Let $\beta>1$, and fix $c>0$ and $V_+>0$. Let $(U,V)$ be a positive
traveling-wave profile satisfying
\eqref{eq:whole-line-band-asymptotics}. Then the profile is unique up to
spatial translation and is given by

\begin{align}
V(s)
&=
V_+
\left[
1+
\exp
\left(
-\frac{c}{D}(s-s_0)
\right)
\right]^{-\frac{1}{\beta-1}},
\label{explicit-V}
\\
U(s)
&=
\frac{c^2V_+}{kD(\beta-1)}
\exp
\left(
-\frac{c}{D}(s-s_0)
\right)
\left[
1+
\exp
\left(
-\frac{c}{D}(s-s_0)
\right)
\right]^{-\frac{\beta}{\beta-1}},
\label{explicit-U}
\end{align}

where $s_0\in\mathbb{R}$ is an arbitrary translation parameter.

\end{Proposition}
\begin{proof}
Combining
\[
V'=\frac{k}{c}U
\]
with the relation obtained in Theorem~\ref{trav-band-case1}, we obtain
\begin{equation}
V'
=
\frac{c}{D(\beta-1)}
V
\left[
1-
\left(
\frac{V}{V_+}
\right)^{\beta-1}
\right].
\label{scalar-V-equation}
\end{equation}
Introduce the normalized variable
\begin{equation}
W
=
\left(
\frac{V}{V_+}
\right)^{\beta-1}.
\end{equation}
Differentiating gives
\[
W'
=
(\beta-1)
\frac{V'}{V}W .
\]
Using \eqref{scalar-V-equation},
\[
W'
=
\frac{c}{D}W(1-W).
\]
Hence the normalized profile satisfies the logistic equation. The heteroclinic
solution connecting $0$ and $1$ is
\begin{equation}
W(s)
=
\frac{1}
{1+
\exp
\left(
-\frac{c}{D}(s-s_0)
\right)} .
\end{equation}
Therefore,
\[
V(s)
=
V_+
W(s)^{\frac{1}{\beta-1}},
\]
which yields \eqref{explicit-V}. The mobile component follows from
\[
U=\frac{c}{k}V',
\]
and gives \eqref{explicit-U}.
\end{proof}

\begin{Theorem}
\label{u-v-montone}
Let $\beta>1$, and let $(U,V)$ be a positive traveling-wave profile satisfying
\eqref{eq:whole-line-band-asymptotics}. Then
\begin{equation}
0<V(s)<V_+,
\qquad
V'(s)>0,
\qquad
U(s)>0,
\qquad s\in\mathbb{R}.
\end{equation}
Moreover, the mobile phase has a unique critical point $s^*$, which is a
strict global maximum. Specifically,
\begin{equation}
U'(s)>0,\qquad s<s^*,
\end{equation}
and
\begin{equation}
U'(s)<0,\qquad s>s^*.
\end{equation}
\end{Theorem}
\begin{proof}
By Proposition~\ref{explicit-profile},
\[
V(s)
=
V_+
\left[
1+
\exp\left(
-\frac{c}{D}(s-s_0)
\right)
\right]^{-1/(\beta-1)}.
\]
Therefore,
\[
0<V(s)<V_+,
\qquad
V'(s)>0,
\]
for every finite $s$. Since
\[
U(s)=\frac{c}{k}V'(s),
\]
it follows that
\[
U(s)>0.
\]

Define
\begin{equation}
F(V)
=
\frac{c^2}{kD(\beta-1)}
V
\left[
1-
\left(
\frac{V}{V_+}
\right)^{\beta-1}
\right].
\end{equation}

By Theorem~\ref{trav-band-case1},
\[
U(s)=F(V(s)).
\]
Hence,
\[
U'(s)=F'(V)V'.
\]
A direct differentiation gives
\[
F'(V)
=
\frac{c^2}{kD(\beta-1)}
\left[
1-\beta
\left(
\frac{V}{V_+}
\right)^{\beta-1}
\right].
\]
Because $V'(s)>0$, the sign of $U'$ is determined by $F'(V)$. The critical
point satisfies
\[
1-\beta
\left(
\frac{V}{V_+}
\right)^{\beta-1}=0,
\]
and therefore
\[
V(s^*)=
V_+\beta^{-\frac1{\beta-1}} .
\]
Since $V$ is strictly monotone, this point is unique. Consequently,
\[
U'(s)>0,\qquad s<s^*,
\]
and
\[
U'(s)<0,\qquad s>s^* .
\]
Therefore the mobile phase forms a single localized hump.
\end{proof}

The explicit representation shows that the traveling band is localized. In
particular, the background phase approaches two different asymptotic states,
while the mobile phase decays at both spatial infinities.
\begin{Corollary}[Asymptotic behavior of the profile]
\label{cor:profile-limits}
The traveling-wave solution satisfies
\begin{equation}
\lim_{s\to-\infty}V(s)=0,
\qquad
\lim_{s\to+\infty}V(s)=V_+ ,
\end{equation}

and

\begin{equation}
\lim_{s\to-\infty}U(s)=0,
\qquad
\lim_{s\to+\infty}U(s)=0 .
\end{equation}

Moreover,

\begin{equation}
V'(s)>0,
\qquad s\in\mathbb{R}.
\end{equation}

\end{Corollary}

\begin{proof}

The limits follow directly from the explicit formulas. Since

\[
V'=\frac{k}{c}U
\]

and $U>0$, we have

\[
V'(s)>0 .
\]

Thus the background profile is strictly increasing.

\end{proof}

{ 
\begin{Remark}[Wave speed and total mobile mass]
\label{remark:wave-speed}
The traveling-wave construction above yields a family of profiles
parameterized by $c>0$; in particular, the asymptotic conditions alone do not
produce a KPP-type minimal-speed condition. The speed is nevertheless related
to the total mass of the mobile component. Indeed, integrating
\eqref{Vprime-basic} over $\mathbb{R}$ gives
\begin{equation}
    V_+-V_-
    =
    \frac{k}{c}
    \int_{\mathbb{R}}U(s)\,ds .
\end{equation}
Hence, if
\begin{equation}
    M
    :=
    \int_{\mathbb{R}}U(s)\,ds,
\end{equation}
then
\begin{equation}
    c
    =
    \frac{kM}{V_+-V_-}.
    \label{eq:speed-mass-relation}
\end{equation}
For the traveling bands considered here, $V_-=0$, and therefore
\begin{equation}
    c=\frac{kM}{V_+}.
\end{equation}
Thus $c$ is free when the total mobile mass is not prescribed, with different
values of $c$ corresponding to traveling bands of different mass. If the
total mobile mass $M$ and the asymptotic background states are prescribed,
however, the wave speed is fixed by \eqref{eq:speed-mass-relation}.

Moreover, under the spatial decay conditions of the traveling band, the
mobile-phase equation \eqref{model-u} is in conservative form, and therefore
\begin{equation}
    \frac{d}{dt}
    \int_{\mathbb{R}}u(x,t)\,dx
    =0.
\end{equation}
Consequently, for an evolution with prescribed initial mobile mass, any
traveling band arising with the stated asymptotic background states must
satisfy \eqref{eq:speed-mass-relation}. Establishing dynamical convergence to
a particular member of the traveling-wave family is a separate stability and
selection problem and is not addressed by the present existence analysis.
\end{Remark}
}

The localization property of the mobile component is quantified by the next
result, which identifies the unique concentration maximum.
\begin{Proposition}[Unique maximum of the mobile component]
\label{prop:maximum-location}

Let $(U,V)$ be the traveling profile given by
Proposition~\ref{explicit-profile}. Then $U$ has a unique critical point
$s^*$, given by
\begin{equation}
s^*
=
s_0-\frac{D}{c}\log(\beta-1).
\end{equation}

This point is a strict global maximum. The maximum value is

\begin{equation}
U^*
=
\frac{c^2V_+}
{kD\,
\beta^{\frac{\beta}{\beta-1}}}.
\end{equation}
\end{Proposition}

\begin{proof}

Set

\[
z=
\exp
\left(
-\frac{c}{D}(s-s_0)
\right).
\]

Then

\[
U(s)
=
\frac{c^2V_+}{kD(\beta-1)}
z(1+z)^{-\frac{\beta}{\beta-1}} .
\]

Since $z$ is strictly decreasing in $s$, critical points of $U$ correspond to
critical points of

\[
G(z)=z(1+z)^{-\frac{\beta}{\beta-1}} .
\]

Differentiation gives

\[
G'(z)
=
(1+z)^{-\frac{\beta}{\beta-1}}
\left[ 
1-\frac{\beta}{\beta-1}
\frac{z}{1+z}
\right].
\]

Hence

\[
G'(z)=0
\]

if and only if

\[
z=\beta-1 .
\]

Because $z$ is monotone, this critical point is unique. Therefore,

\[
s^*
=
s_0-\frac{D}{c}\log(\beta-1).
\]

The sign of $G'$ changes from positive to negative at this point, proving that
the critical point is a strict global maximum.

Substitution of $z=\beta-1$ into the explicit expression gives the stated
value of $U^*$.

\end{proof}

The previous results establish the complete qualitative structure of the
traveling band. The background component forms a monotone transition layer,
while the mobile component develops a localized pulse with a single maximum.
\begin{Theorem}[Existence of a supercritical one-hump traveling band]
\label{thm:strong-band}
Assume
\begin{equation}
\beta>1,
\qquad
k>0,
\qquad
D>0 .
\end{equation}
Then the system admits a traveling-wave solution
\begin{equation}
(u,v)(x,t)
=
(U,V)(x-ct)
\end{equation}
such that
\begin{equation}
V(-\infty)=0,
\qquad
V(+\infty)=V_+,
\end{equation}
and
\begin{equation}
U(-\infty)=U(+\infty)=0 .
\end{equation}
The background component is strictly increasing, and the mobile component
forms a strong one-hump profile. { For fixed $c>0$ and $V_+>0$,
the traveling-wave profile is unique up to spatial translation.}
\end{Theorem}

\begin{proof}
The explicit formulas from Proposition~\ref{explicit-profile} satisfy the
traveling-wave system and the required asymptotic conditions. The monotonicity
of $V$ follows from Corollary~\ref{cor:profile-limits}, and the single-hump
property of $U$ follows from Proposition~\ref{prop:maximum-location}.

{ 
For fixed $c>0$ and $V_+>0$, the only remaining free parameter in the
explicit profile is the shift $s_0$, which reflects the translation
invariance of the traveling-wave equations.
}
\end{proof}

\section{Linearized Perturbation Analysis of the Traveling-Band Solution}
\label{par4}

The traveling-band solutions constructed in Section~\ref{par3} provide explicit
coherent structures arising from the interaction of diffusion,
gradient-induced transport, and background depletion. The existence of these
profiles characterizes the formation mechanism of the band, while the response
of the structure to perturbations is governed by the associated linearized
dynamics. We therefore analyze the evolution of small perturbations around the
traveling-wave state.
The objective of this section is to derive the linearized perturbation system,
identify the resulting variable-coefficient operator, and determine the
behavior of its coefficients in the far-field regions. 
The explicit structure of the traveling profile, together with a relative
smallness condition on the background perturbation, allows us to control the
singular behavior generated by the nonlinear transport term and to
{ establish a conditional finite-time a priori estimate for the
mobile-phase perturbation under bounded background forcing through a
maximum-principle argument.}
These results
provide the appropriate linear framework for future spectral and nonlinear
stability analyses.

Let
\[
s=x-ct,
\]
and denote the traveling-band profile constructed in Section~\ref{par3} by
$(U,V)$. In this section, we consider the supercritical traveling-band profile
corresponding to $\beta>1$.
The traveling-band profile satisfies
\[
U(s)>0,
\qquad
V(s)>0,
\]
for every $s\in\mathbb{R}$, together with the structural identity
\begin{equation}
V'(s)=\frac{k}{c}U(s).
\label{Vprime-linear}
\end{equation}
Moreover, the profile connects the asymptotic states
\begin{equation}
(U,V)(-\infty)=(0,0),
\end{equation}
and
\begin{equation}
(U,V)(+\infty)=(0,V_+),
\qquad V_+>0 .
\end{equation}
The two limiting states have different transport characteristics. In particular,
the background component approaches a depleted state at one spatial end and a
positive equilibrium state at the other. Since the linearized transport
operator contains ratios involving $V$, the asymptotic behavior of these
coefficients must be established before applying standard parabolic estimates.
We introduce perturbations in the traveling coordinate by writing
\begin{equation}
u(x,t)=U(s)+\bar u(s,t),
\qquad
v(x,t)=V(s)+\bar v(s,t),
\label{linear-perturbation-decomposition}
\end{equation}
where
\[
s=x-ct .
\]

The functions $\bar u$ and $\bar v$ denote perturbations of the mobile and
background components, respectively.  { Since
\[
V(s)\to0
\qquad\text{as }s\to-\infty,
\]
absolute smallness of $\bar v$ does not ensure that the perturbed background
$V+\bar v$ remains positive or that the nonlinear transport term can be
linearized uniformly. We therefore impose the relative smallness condition
\begin{equation}
    |\bar v(s,t)|
    \leq
    \theta V(s),
    \qquad
    (s,t)\in\mathbb{R}\times[0,T],
    \qquad
    0<\theta<1.
    \label{relative-background-perturbation}
\end{equation}
Consequently, for every finite $s$,
\begin{equation}
    \left|\frac{\bar v(s,t)}{V(s)}\right|
    \leq\theta<1,
    \qquad
    V(s)+\bar v(s,t)
    \geq
    (1-\theta)V(s)>0.
    \label{positive-perturbed-background}
\end{equation}
Thus the perturbation remains controlled relative to
the traveling background, including in the depleted far field.}

 {Substituting \eqref{linear-perturbation-decomposition} into the governing
system,
\begin{equation}
u_t
=
-\chi
\left(
u\frac{v_x}{v}
\right)_x
+
Du_{xx},
\end{equation}
and
\begin{equation}
v_t=-ku,
\end{equation}
we linearize the nonlinear transport flux about the traveling profile.
The derivatives in the moving frame satisfy
\begin{equation}
u_t=-cU'(s)+\bar u_t-c\bar u_s.
\end{equation}
Note that 
\begin{equation}
v_x=V'(s)+\bar v_s.
\end{equation}
Hence,
\begin{equation}
u\frac{v_x}{v}
=
(U+\bar u)
\frac{V'+\bar v_s}{V+\bar v}.
\end{equation}
By \eqref{relative-background-perturbation},
\begin{equation}
\frac{1}{V+\bar v}
=
\frac{1}{V}
\left(
1+\frac{\bar v}{V}
\right)^{-1}.
\end{equation}
Since
\[
\left|\frac{\bar v}{V}\right|
\leq\theta<1,
\]
the denominator admits an expansion in the relative perturbation
$\bar v/V$:
\begin{equation}
\frac{1}{V+\bar v}
=
\frac{1}{V}
-
\frac{\bar v}{V^2}
+
\frac{1}{V}
O\left(
\left|\frac{\bar v}{V}\right|^2
\right).
\label{relative-denominator-expansion}
\end{equation}
Retaining the terms that are linear in the perturbations yields
\begin{equation}
u\frac{v_x}{v}
=
\frac{UV'}{V}
+
\frac{V'}{V}\bar u
+
\frac{U}{V}\bar v_s
-
\frac{UV'}{V^2}\bar v
+
\text{higher-order terms}.
\label{transport-expansion}
\end{equation}
The leading-order term
\begin{equation}
\frac{UV'}{V}
\end{equation}
is exactly the transport contribution appearing in the traveling-wave equation.
Therefore, this term cancels after subtraction of the profile equation. The
remaining first-order terms determine the perturbation dynamics:
\begin{equation}
\bar u_t-c\bar u_s
=
D\bar u_{ss}
-
\chi
\partial_s
\left(
\frac{V'}{V}\bar u
+
\frac{U}{V}\bar v_s
-
\frac{UV'}{V^2}\bar v
\right).
\label{linear-u-revised}
\end{equation}
For the background component, subtracting the traveling-wave equation
\[
-cV'=-kU
\]
from the full background equation gives
\begin{equation}
\bar v_t-c\bar v_s
=
-k\bar u .
\label{linear-v-revised}
\end{equation}
Consequently, the perturbation system around the traveling band is
\begin{align}
\bar u_t-c\bar u_s
&=
D\bar u_{ss}
-
\chi
\partial_s
\left(
\frac{V'}{V}\bar u
+
\frac{U}{V}\bar v_s
-
\frac{UV'}{V^2}\bar v
\right),
\label{linear-system-u-revised}
\\
\bar v_t-c\bar v_s
&=
-k\bar u .
\label{linear-system-v-revised}
\end{align}
Introducing
\[
W=
\begin{pmatrix}
\bar u\\
\bar v
\end{pmatrix},
\]
the system takes the abstract form
\begin{equation}
W_t=\mathcal{\rm L}_{\rm lin}W ,
\label{abstract-linear-operator}
\end{equation}
where $\mathcal{\rm L}_{\rm lin}$ is the variable-coefficient linearized operator
induced by the traveling-band profile.
The remainder of this section is devoted to the analysis of
$\mathcal{\rm L}_{\rm lin}$. In particular, we establish the boundedness and
far-field behavior of its coefficients, which provides the basis for a
finite-time perturbation estimate.}

The linearized operator contains the coefficient ratios
\[
\frac{V'}{V},
\qquad
\frac{U}{V},
\qquad
\frac{UV'}{V^2}.
\]
Although these coefficients appear singular near the depleted state, the
explicit traveling-band profile provides a precise balance between the mobile
and background components that controls the apparent singularity. The relation
\[
V'=\frac{k}{c}U
\]
plays a central role in this cancellation mechanism. The following Proposition
establishes the limiting behavior of these coefficients and verifies the
boundedness properties required for the linearized perturbation analysis.
\begin{Proposition}[Far-field behavior of the linearized coefficients]
\label{linear-coeff-asymp-revised}
Let $(U,V)$ be the traveling-band profile constructed in
Section~\ref{par3}. Then
\begin{equation}
\begin{aligned}
\frac{V'(s)}{V(s)}
    &\to 0,
    && s\to +\infty,\\
\frac{V'(s)}{V(s)}
    &\to \frac{c}{D(\beta-1)},
    && s\to -\infty .
\end{aligned}
\end{equation}
and
\begin{equation}
\begin{aligned}
\frac{U(s)}{V(s)}
    &\to 0,
    && s\to +\infty,\\
\frac{U(s)}{V(s)}
    &\to \frac{c^2}{kD(\beta-1)},
    && s\to -\infty .
\end{aligned}
\end{equation}
Consequently,
\begin{equation}
\frac{UV'}{V^2}
\end{equation}
is bounded on $\mathbb R$ and admits finite limits at both spatial
infinities.
\end{Proposition}

\begin{proof}
From the explicit representation of the traveling profile obtained in
Proposition~\ref{explicit-profile},
\[
V(s)
=
V_+
\left(
1+
e^{-\frac{c}{D}(s-s_0)}
\right)^{-\frac1{\beta-1}},
\]
we compute
\[
\frac{V'(s)}{V(s)}
=
\frac{c}{D(\beta-1)}
\frac{
e^{-\frac{c}{D}(s-s_0)}
}{
1+
e^{-\frac{c}{D}(s-s_0)}
}.
\]
Hence,
\[
\frac{V'(s)}{V(s)}
\to0,
\qquad
s\to+\infty ,
\]
and
\[
\frac{V'(s)}{V(s)}
\to
\frac{c}{D(\beta-1)},
\qquad
s\to-\infty .
\]
Using the traveling-wave identity
\[
V'=\frac{k}{c}U ,
\]
we have
\[
\frac{U}{V}
=
\frac{c}{k}\frac{V'}{V}.
\]
Therefore,
\[
\frac{U}{V}\to0,
\qquad
s\to+\infty ,
\]
while
\[
\frac{U}{V}
\to
\frac{c^2}{kD(\beta-1)},
\qquad
s\to-\infty .
\]
Finally,
\[
\frac{UV'}{V^2}
=
\frac{U}{V}
\frac{V'}{V}.
\]
Both factors on the right-hand side have finite limits at
$\pm\infty$ and are bounded by the explicit traveling profile.
Therefore,
\[
\frac{UV'}{V^2}
\]
remains bounded on $\mathbb R$ and possesses finite limiting values.
\end{proof}

The previous Proposition shows that the apparent singular behavior of the
linearized coefficients near the depleted side is controlled by the structure
of the traveling profile. In particular, the linearized operator is
well-defined, and its coefficients approach finite limiting states at both
spatial infinities. We now rewrite the perturbation system in an explicit variable-coefficient
form. Expanding the divergence term in the mobile-phase equation gives

\begin{equation}
\bar u_t
=
D\bar u_{ss}
+
\lambda_2(s)\bar u_s
+
\lambda_1(s)\bar u
+
\lambda_3(s)\bar v
+
\lambda_4(s)\bar v_s
+
\lambda_5(s)\bar v_{ss},
\label{expanded-linear-system-revised}
\end{equation}
where the variable coefficients $\lambda_i(s)$ are defined by
\[
\begin{aligned}
\lambda_1(s)&=-\chi\left(\frac{V'}{V}\right)',&
\lambda_2(s)&=c-\chi\frac{V'}{V},&
\lambda_3(s)&=\chi\left(\frac{UV'}{V^2}\right)',\\
\lambda_4(s)&=-\chi\left(\frac{U'}{V}-2\frac{UV'}{V^2}\right),&
\lambda_5(s)&=-\chi\frac{U}{V}.
\end{aligned}
\]

\begin{Remark}
The non-divergence form of the linearized equation is particularly suitable
for the maximum-principle argument employed below, which yields a direct
finite-time $\rm{L}^\infty$ estimate for the mobile-phase perturbation.
\end{Remark}

The background perturbation satisfies
\begin{equation}
\bar v_t-c\bar v_s=-k\bar u .
\label{expanded-v-equation-revised}
\end{equation}
Hence the mobile-phase equation is a uniformly parabolic equation with
principal coefficient $D>0$, while all lower-order coefficients are determined
by the traveling-band profile.
The following Proposition describes the limiting structure of the linearized
operator at the two far-field states.

{ 
\begin{Proposition}[Asymptotic structure of the linearized operator]
\label{lambda-asymptotic-revised}
The coefficient functions defined in
\eqref{expanded-linear-system-revised} have finite limits at both spatial
infinities. As $s\to+\infty$,
\begin{equation}
\lambda_1(s)\to0,\qquad
\lambda_2(s)\to c,\qquad
\lambda_3(s)\to0,\qquad
\lambda_4(s)\to0,\qquad
\lambda_5(s)\to0.
\label{eq:lambda-limits-plus}
\end{equation}
As $s\to-\infty$,
\begin{equation}
\begin{aligned}
\lambda_1(s)&\to0,
&
\lambda_2(s)&\to-\frac{c}{\beta-1},
&
\lambda_3(s)&\to0,
\\
\lambda_4(s)&\to
\frac{\chi c^3}{kD^2(\beta-1)^2},
&
\lambda_5(s)&\to
-\frac{\chi c^2}{kD(\beta-1)} .
\end{aligned}
\label{eq:lambda-limits-minus}
\end{equation}
Consequently, the linearized operator approaches constant-coefficient
operators at both spatial infinities.
\end{Proposition}
}

\begin{proof}
{  
We first consider the limit as $s\to+\infty$. From the explicit
traveling-wave representation,
\[
\frac{V'}{V}\to0,
\qquad
\frac{U}{V}\to0,
\qquad
s\to+\infty.
\]
The corresponding derivatives also vanish. Therefore,
\[
\lambda_1(s), \ \lambda_3(s), \ \lambda_4(s), \ \lambda_5(s)\to0,
\qquad
s\to+\infty,
\]
while
\[
\lambda_2(s)
=
c-\chi\frac{V'}{V}
\longrightarrow c.
\]
Thus the limiting mobile-phase operator at the right far field contains
the constant convective term $c\,\partial_s$.
}

We now consider the depleted side. From
Proposition~\ref{linear-coeff-asymp-revised},

\[
\frac{V'}{V}
\rightarrow
\frac{c}{D(\beta-1)},
\]

and

\[
\frac{U}{V}
\rightarrow
\frac{c^2}{kD(\beta-1)} .
\]

{ 
The coefficient $\lambda_2(s)$ satisfies
\[
\lambda_2(s)
=
c-\chi\frac{V'}{V}.
\]
Therefore,
\[
\lambda_2 (s)
\longrightarrow
c-\frac{\chi c}{D(\beta-1)}.
\]
Since $\beta=\chi/D$, it follows that
\[
\lambda_2 (s)
\longrightarrow
c-\frac{\beta c}{\beta-1}
=
-\frac{c}{\beta-1}.
\]
}

Similarly,

\[
\lambda_5 (s)
=
-\chi\frac{U}{V}
\rightarrow
-\frac{\chi c^2}{kD(\beta-1)} .
\]

For $\lambda_4$, we use the identity

\[
\lambda_4 (s)
=
-\chi
\left(
\frac{U'}{V}
-
2\frac{UV'}{V^2}
\right).
\]

From the traveling-wave relation derived in Section~\ref{par3},

\[
\frac{U'}{U}
=
\beta\frac{V'}{V}-\frac{c}{D}.
\]

Hence,

\[
\frac{U'}{V}
=
\frac{U}{V}
\left(
\beta\frac{V'}{V}
-\frac{c}{D}
\right).
\]

Taking the limit as $s\to-\infty$ gives

\[
\frac{U'}{V}
\rightarrow
\frac{c^2}{kD(\beta-1)}
\left(
\frac{\beta c}{D(\beta-1)}
-\frac{c}{D}
\right),
\]

and therefore

\[
\frac{U'}{V}
\rightarrow
\frac{c^3}{kD^2(\beta-1)^2}.
\]

Also,

\[
\frac{UV'}{V^2}
=
\frac{U}{V}\frac{V'}{V}
\rightarrow
\frac{c^3}{kD^2(\beta-1)^2}.
\]

Consequently,

\[
\lambda_4 (s)
\rightarrow
-\chi
\left(
\frac{c^3}{kD^2(\beta-1)^2}
-
2\frac{c^3}{kD^2(\beta-1)^2}
\right),
\]

which yields

\[
\lambda_4 (s)
\rightarrow
\frac{\chi c^3}
{kD^2(\beta-1)^2}.
\]

Finally, since

\[
\frac{V'}{V}
\rightarrow
\frac{c}{D(\beta-1)}
\]

is constant at leading order, we have

\[
\left(\frac{V'}{V}\right)'
\rightarrow0 ,
\]

and therefore

\[
\lambda_1(s)\rightarrow0 .
\]

The same far-field expansion gives

\[
\left(
\frac{UV'}{V^2}
\right)'
\rightarrow0 ,
\]

which implies

\[
\lambda_3(s)\rightarrow0 .
\]

Thus every coefficient admits a finite limiting value, and the linearized
operator approaches constant-coefficient operators at both spatial ends.

\end{proof}

The asymptotic analysis above, together with the explicit traveling-wave representation, shows that the coefficients of the linearized operator are uniformly bounded on $\mathbb{R}$. 

We now use the
parabolic structure of the mobile-phase equation to { derive a
conditional finite-time a priori estimate for the mobile-phase perturbation
under bounded background forcing.}
The first equation of the linearized system can be written as
\begin{equation}
\bar u_t-\mathbf{L}\bar u
=
\lambda_3(s)\bar v
+
\lambda_4(s)\bar v_s
+
\lambda_5(s)\bar v_{ss},
\label{eq:linear-forcing}
\end{equation}
where
\begin{equation}
\mathbf{L}\bar u
=
D\bar u_{ss}
+
\lambda_2(s)\bar u_s
+
\lambda_1(s)\bar u .
\label{L-op-revised}
\end{equation}

{ 

The diffusion coefficient satisfies
\[
D>0,
\]
and therefore $\mathbf L$ is uniformly parabolic. From the explicit
traveling-band representation,
\[
\frac{V'}{V}
=
\frac{c}{D(\beta-1)}
\frac{
e^{-\frac{c}{D}(s-s_0)}
}{
1+e^{-\frac{c}{D}(s-s_0)}
},
\]
and hence
\[
\left(\frac{V'}{V}\right)'
=
-\frac{c^2}{D^2(\beta-1)}
\frac{
e^{-\frac{c}{D}(s-s_0)}
}{
\left(1+e^{-\frac{c}{D}(s-s_0)}\right)^2
}.
\]
Consequently,
\[
\lambda_1(s)
=
-\chi\left(\frac{V'}{V}\right)'
=
\frac{\chi c^2}{D^2(\beta-1)}
\frac{
e^{-\frac{c}{D}(s-s_0)}
}{
\left(1+e^{-\frac{c}{D}(s-s_0)}\right)^2
}
\geq0.
\]
Since
\[
0\leq
\frac{z}{(1+z)^2}
\leq\frac14,
\qquad z\geq0,
\]
we obtain the uniform bound
\begin{equation}
0\leq\lambda_1(s)\leq
\frac{\chi c^2}{4D^2(\beta-1)}
=:\Lambda_1.
\label{lambda1-uniform-bound}
\end{equation}
Thus the zeroth-order coefficient does not have the nonpositive sign
condition directly. Introducing the exponentially weighted perturbation
\[
w(s,t)=e^{-\Lambda_1t}\bar u(s,t)
\]
replaces the zeroth-order coefficient by
$\lambda_1(s)-\Lambda_1\leq0$, allowing the parabolic maximum principle
to be applied to the transformed equation.
}
\\

The remaining terms arise from the coupling with the background perturbation
and its spatial derivatives. They enter the mobile-phase equation as
lower-order forcing terms. { 
We now derive a conditional finite-time a priori estimate for the
mobile-phase perturbation under prescribed bounds on these forcing terms.
}

\begin{Theorem}[Finite-time a priori estimate under bounded background forcing]
\label{thm:max-principle}
Let $\beta>1$, $D>0$, and
\[
\Omega_L=(-L,L)\times(0,T).
\]
Suppose that
\[
\bar u\in C(\overline{\Omega_L})\cap C^{2,1}(\Omega_L)
\]
satisfies
\begin{equation}
\bar u_t
=
D\bar u_{ss}
+
\lambda_2(s)\bar u_s
+
\lambda_1(s)\bar u
+
\lambda_3(s)\bar v
+
\lambda_4(s)\bar v_s
+
\lambda_5(s)\bar v_{ss},
\label{eq:linear-u-max}
\end{equation}
with
\begin{equation}
\bar u(s,0)=\bar u_0(s),
\qquad
\bar u(\pm L,t)=0.
\end{equation}
Assume that
\begin{equation}
|\bar v|\leq V_0,
\qquad
|\bar v_s|\leq V_1,
\qquad
|\bar v_{ss}|\leq V_2
\qquad\text{in }\Omega_L.
\end{equation}

{ 
Thus, the background forcing term
\[
F(s,t)
:=
\lambda_3(s)\bar v(s,t)
+
\lambda_4(s)\bar v_s(s,t)
+
\lambda_5(s)\bar v_{ss}(s,t)
\]
is uniformly bounded in $\Omega_L$.
}

Define
\begin{equation}
M
:=
\sup_{s\in\mathbb R}\lambda_1(s)
=
\frac{\chi c^2}{4D^2(\beta-1)}
\end{equation}
and
\begin{equation}
K
:=
\|\lambda_3\|_{\mathrm{L}^\infty(\mathbb{R})}V_0
+
\|\lambda_4\|_{\mathrm{L}^\infty(\mathbb{R})}V_1
+
\|\lambda_5\|_{\mathrm{L}^\infty(\mathbb{R})}V_2.
\label{eq:K-definition}
\end{equation}
Then, for every $0\leq t\leq T$,
\begin{equation}
\|\bar u(\cdot,t)\|_{\mathrm{L}^\infty(-L,L)}
\leq
e^{Mt}
\left(
\|\bar u_0\|_{\mathrm{L}^\infty(-L,L)}
+
Kt
\right).
\label{eq:max-principle-estimate}
\end{equation}
Consequently,
\begin{equation}
\|\bar u\|_{\mathrm{L}^\infty(\Omega_L)}
\leq
e^{MT}
\left(
\|\bar u_0\|_{\rm{L} ^\infty(-L,L)}
+
KT
\right).
\end{equation}
\end{Theorem}

\begin{proof}
Write \eqref{eq:linear-u-max} as
\begin{equation}
\bar u_t
=
D\bar u_{ss}
+
\lambda_2(s)\bar u_s
+
\lambda_1(s)\bar u
+
F(s,t),
\label{eq:linear-u-forcing}
\end{equation}
where
\begin{equation}
F
=
\lambda_3\bar v
+
\lambda_4\bar v_s
+
\lambda_5\bar v_{ss}.
\end{equation}

{ 
By the assumed bounds on $\bar v$, $\bar v_s$, and $\bar v_{ss}$ and the
definition of $K$ in \eqref{eq:K-definition},
\begin{equation}
|F(s,t)|
\leq K,
\qquad
(s,t)\in\Omega_L.
\label{eq:F-bound}
\end{equation}
}

By \eqref{lambda1-uniform-bound},
\[
0\leq\lambda_1(s)\leq M,
\qquad
M=\frac{\chi c^2}{4D^2(\beta-1)}.
\]

Introduce the exponentially weighted perturbation
\begin{equation}
w(s,t)=e^{-Mt}\bar u(s,t).
\label{eq:w-transform}
\end{equation}
Substitution into \eqref{eq:linear-u-forcing} gives
\begin{equation}
w_t
=
Dw_{ss}
+
\lambda_2(s)w_s
+
\bigl(\lambda_1(s)-M\bigr)w
+
e^{-Mt}F(s,t).
\label{eq:w-equation}
\end{equation}
Define
\begin{equation}
\widetilde{\mathbf L}w
=
Dw_{ss}
+
\lambda_2(s)w_s
+
\bigl(\lambda_1(s)-M\bigr)w.
\end{equation}
Since $D>0$, $\widetilde{\mathbf L}$ is uniformly parabolic, and
\begin{equation}
\lambda_1(s)-M\leq0.
\end{equation}
Thus the zeroth-order coefficient satisfies the sign condition required by
the parabolic maximum principle.

Moreover, using \eqref{eq:F-bound},
\begin{equation}
\left|
\widetilde{\mathbf L}w-w_t
\right|
=
e^{-Mt}|F(s,t)|
\leq K.
\end{equation}
The transformed initial and boundary data satisfy
\begin{equation}
w(s,0)=\bar u_0(s),
\qquad
w(\pm L,t)=0.
\end{equation}
Therefore, the maximum-principle estimate
\cite[Exercise~2.1, p.~117]{Landis} gives
\begin{equation}
|w(s,t)|
\leq
\|\bar u_0\|_{\mathrm{L}^\infty(-L,L)}
+
Kt.
\end{equation}
Since
\[
\bar u(s,t)=e^{Mt}w(s,t),
\]
it follows that
\begin{equation}
|\bar u(s,t)|
\leq
e^{Mt}
\left(
\|\bar u_0\|_{\mathrm{L}^\infty(-L,L)}
+
Kt
\right).
\end{equation}
Taking the supremum over $s\in(-L,L)$ gives
\eqref{eq:max-principle-estimate}, and then taking the supremum over
$0\leq t\leq T$ yields
\[
\|\bar u\|_{\mathrm{L}^\infty(\Omega_L)}
\leq
e^{MT}
\left(
\|\bar u_0\|_{\mathrm{L}^\infty(-L,L)}
+
KT
\right).
\]
This completes the proof.
\end{proof}

{ 
The estimate in Theorem~\ref{thm:max-principle} is conditional on the
prescribed bounds for $\bar v$, $\bar v_s$, and $\bar v_{ss}$. These bounds
are assumed in the theorem and are not derived from the initial perturbations
through the coupled linearized system. Accordingly, the result is a
finite-time a priori estimate for the mobile-phase perturbation $\bar u$
under bounded background forcing, rather than a closed stability estimate
for the full coupled perturbation system.
}

{ 
The previous estimate provides finite-time control of the mobile-phase
perturbation under bounded background forcing. The background equation
admits a characteristic representation, which provides a direct bound for
$\bar v$ in terms of its initial value and $\bar u$. It does not, however,
yield corresponding bounds for $\bar v_s$ and $\bar v_{ss}$ from the
$\mathrm{L}^\infty$ estimate for $\bar u$ alone. Consequently, the bounds on
$\bar v_s$ and $\bar v_{ss}$ required in
Theorem~\ref{thm:max-principle} remain additional assumptions, and the
maximum-principle estimate does not by itself close the full coupled
linearized perturbation system.}

The analysis in this section establishes the linear perturbation framework
associated with the traveling-band solution. The explicit structure of the
traveling profile removes the apparent singular behavior of the transport
coefficients and yields a well-defined variable-coefficient operator with
controlled far-field behavior.
 {
Theorem~\ref{thm:max-principle} provides a direct connection between the
maximum principle and quantitative bounds on the solution. The
maximum-principle argument yields a conditional finite-time a priori estimate
for the mobile-phase perturbation under the prescribed background forcing
bounds, but does not establish asymptotic stability of the full coupled
perturbation system. Accordingly, these results do not constitute a complete
linear stability analysis and do not establish asymptotic decay as
$t\rightarrow\infty$.
}

A complete long-time stability analysis requires additional spectral
information on the linearized operator, including the treatment of the neutral
translational mode generated by the invariance of the traveling-band profile
under spatial translations.
In the next section, we return to the traveling-wave equations and derive a
reduced scalar characterization of the traveling-band profile.

\section {An Alternative Characterization of the Traveling-Band Profile}
\label{par5}

The traveling-wave analysis of Section~\ref{par3} and the perturbation
framework of Section~\ref{par4} provide the two main components of the
traveling-band description. In this section, we derive a reduced scalar
formulation of the profile equations by eliminating the mobile component from
the coupled system. The reduction is obtained directly from the traveling-wave equations. Since
the two phases remain coupled through the background evolution equation, the
mobile component of a traveling-wave solution is determined by the derivative
of the background profile. Consequently, the two-component profile system can
be reduced to a single nonlinear third-order equation for the background
phase.

\begin{Theorem}[Reduced scalar equation for the traveling-band profile]
\label{thm:third-order-equation}
Let $(U,V)$ be a traveling-wave profile in the coordinate
$s=x-ct$ satisfying the profile system
 \eqref{ode-system-one}--\eqref{ode-system-two}. Assume that $k\neq0$. Then the background component $V$ satisfies the scalar equation
\begin{equation}
V'''
+
\frac{c}{D}V''
-
\beta
\left[
\frac{(V')^2}{V}
\right]'
=0 .
\label{3-order-V-ode}
\end{equation}

Moreover, the explicit traveling-band profile obtained in
Proposition~\ref{explicit-profile} satisfies this reduced scalar equation.
\end{Theorem}

\begin{proof}
The second traveling-wave equation provides the relation between the mobile
and background phases,
\[
cV'=kU .
\]
Since $k\neq0$, the mobile component can be expressed as
\begin{equation}
U=\frac{c}{k}V',
\qquad
U'=\frac{c}{k}V'',
\qquad
U''=\frac{c}{k}V''' .
\end{equation}

We substitute these identities into the first traveling-wave equation.

The nonlinear transport term becomes
\[
\frac{\chi}{c}
\left[
U\frac{V'}{V}
\right]'
=
\frac{\chi}{c}
\left[
\frac{c}{k}\frac{(V')^2}{V}
\right]'
=
\frac{\chi}{k}
\left[
\frac{(V')^2}{V}
\right]' .
\]
The diffusion contribution satisfies
\[
\frac{D}{c}U''
=
\frac{D}{k}V''' .
\]
Hence the first profile equation reduces to
\[
\frac{c}{k}V''
=
\frac{\chi}{k}
\left[ 
\frac{(V')^2}{V}
\right]'
-
\frac{D}{k}V''' .
\]
Multiplying by $k/D$ gives
\[
\frac{c}{D}V''
=
\frac{\chi}{D}
\left(
\frac{(V')^2}{V}
\right)'
-
V''' .
\]
Using the dimensionless parameter
\[
\beta=\frac{\chi}{D},
\]
we obtain
\[
\frac{c}{D}V''
=
\beta
\left[ 
\frac{(V')^2}{V}
\right]'
-
V''' .
\]
Rearranging yields
\[
V'''
+
\frac{c}{D}V''
-
\beta
\left[ 
\frac{(V')^2}{V}
\right]'
=0 ,
\]
which proves the reduced scalar equation.

The explicit profile $V$ obtained in Proposition~\ref{explicit-profile} is
consistent with this reduction. Define
\begin{equation}
Z(s)
=
\left(
\frac{V(s)}{V_+}
\right)^{\beta-1}.
\end{equation}
From the expression for $V$ in Proposition~\ref{explicit-profile}, we obtain
\[
V'
=
\frac{c}{D(\beta-1)}
V(1-Z).
\]
Differentiating gives
\[
V''
=
\left(
\frac{c}{D(\beta-1)}
\right)^2
V(1-Z)(1-\beta Z).
\]
A further differentiation yields
\begin{align}
V'''
&=
\left(
\frac{c}{D(\beta-1)}
\right)^3
V(1-Z)
\left[
1-\beta(\beta+1)Z
+\beta(2\beta-1)Z^2
\right].
\end{align}
On the other hand,
\[
\frac{(V')^2}{V}
=
\left(
\frac{c}{D(\beta-1)}
\right)^2
V(1-Z)^2 ,
\]
and therefore
\begin{align}
\left[ 
\frac{(V')^2}{V}
\right]'
&=
\left(
\frac{c}{D(\beta-1)}
\right)^3
V(1-Z)^2
\left(
1-(2\beta-1)Z
\right).
\end{align}

Substituting these expressions into the nonlinear term gives
\begin{align}
&\beta
\left[ 
\frac{(V')^2}{V}
\right]'
-
V'''
\nonumber\\
&=
\left(
\frac{c}{D(\beta-1)}
\right)^3
V(1-Z)
\Bigg[
\beta(1-Z)
\left(
1-(2\beta-1)Z
\right)
\nonumber\\
&\qquad\qquad
-
\left(
1-\beta(\beta+1)Z
+\beta(2\beta-1)Z^2
\right)
\Bigg].
\end{align}
The expression inside the brackets simplifies to
\[
(\beta-1)(1-\beta Z).
\]
Consequently,
\[
\beta
\left[ 
\frac{(V')^2}{V}
\right]'
-
V'''
=
\frac{c}{D}V'' ,
\]
where the expression for $V''$ has been used.

Therefore,
\[
V'''
+
\frac{c}{D}V''
-
\beta
\left[ 
\frac{(V')^2}{V}
\right]'
=0 .
\]
Hence the explicit traveling-band profile satisfies the reduced scalar
equation.
\end{proof}
Theorem~\ref{thm:third-order-equation} provides a scalar characterization of
the traveling-band profile in terms of the background phase alone. Once the
solution $V$ is determined, the mobile component is recovered through the
relation
\[
U=\frac{c}{k}V'.
\]
Thus the coupled traveling-wave system can be represented by a single
nonlinear scalar equation together with the reconstruction formula for the
mobile phase.

Combined with the results of Section~\ref{par3}, this reduction shows that the
existence and structure of the traveling band are governed by the background
profile.  The analysis of Section~\ref{par4} complements this description by
identifying the variable-coefficient linearized operator associated with the
profile and { establishing a conditional finite-time a priori
estimate for the mobile-phase perturbation under bounded background forcing.}
The reduced scalar formulation provides a convenient starting point for
further investigations of parameter dependence, uniqueness, and long-time
stability properties of traveling-band solutions.

\begin{Theorem}[Summary of the traveling-band structure]
\label{main-traveling-band}
Assume
\begin{equation}
\beta>1,\qquad k>0,\qquad D>0 .
\end{equation}
Then, for fixed $c>0$ and $V_+>0$, the traveling-wave system admits a
one-hump traveling-band profile that is { unique up to spatial
translation} and satisfies
\begin{equation}
(U,V)(-\infty)=(0,0),
\end{equation}
and
\begin{equation}
(U,V)(+\infty)=(0,V_+).
\end{equation}
Moreover, the background component satisfies the reduced scalar equation
\begin{equation}
V'''
+\frac{c}{D}V''
-\beta
\left(
\frac{(V')^2}{V}
\right)'
=0 ,
\end{equation}
and the mobile component is recovered from
\begin{equation}
U=\frac{c}{k}V'.
\end{equation}
The linearized operator around this profile has bounded variable coefficients
and admits controlled far-field limits as established in
Section~\ref{par4}.
\end{Theorem}

\section{Discussion}
\label{sec:discussion}

The analysis developed in this article provides a mathematical framework for
constructing and analyzing localized traveling structures in a two-phase
transport system derived from an Einstein-type material-balance principle.
The model couples a mobile phase with an evolving background phase, and the
transport mechanism results from the interaction between the two components.
This coupling produces a balance between diffusive spreading,
gradient-induced transport, and background depletion, leading to the formation
of coherent traveling profiles. 

 {The traveling-wave analysis shows that the coupled nonlinear system admits
globally positive traveling-band solutions for
\[
\beta=\frac{\chi}{D}\geq 1,
\]
with distinct critical and supercritical profile structures.
The critical case $\beta=1$ admits a distinct double-exponential profile, whereas for $0<\beta<1$ the positive branch reaches the depleted
state at a finite traveling coordinate and therefore does not belong to the
same whole-line positive traveling-wave class.}

The resulting structure consists of a localized mobile-phase component coupled
to a monotone background profile. The explicit profile
representation demonstrates that the traveling band is generated by the
internal coupling of the two phases rather than by an imposed external
velocity field. In particular, the mobile component is recovered from the
background evolution through
\[
U=\frac{c}{k}V',
\]
showing that the spatial localization of the mobile phase is determined by
the variation of the background profile. 

The explicit traveling-wave construction also provides
\textit{without boundary conditions} information about the
geometry (pattern) of the profile. The background component connects two
asymptotic states, while the mobile component forms a localized profile
with a single maximum. This structure reflects the
balance between transport and diffusion in the traveling frame. The profile is
determined by the model parameters together with the asymptotic background
state, up to the natural translation invariance of the traveling coordinate.

A consequence of the traveling-wave reduction is the scalar formulation
derived in Section~\ref{par5}. Eliminating the mobile component reduces the
coupled profile equations to the nonlinear third-order equation
\[
V'''
+
\frac{c}{D}V''
-
\beta
\left(
\frac{(V')^2}{V}
\right)'
=0 .
\]
This equation expresses the traveling-band structure entirely through the
background phase. Once the solution of this scalar problem is determined, the
mobile component is recovered through the relation
$U=(c/k)V'$. The reduced formulation therefore provides an alternative
description of the traveling-wave problem and may be useful for further
analysis of parameter dependence and qualitative properties of the profile.
The perturbation analysis around the traveling profile establishes the
corresponding linearized evolution problem. The explicit form of the
traveling wave is important in this setting because it controls the
$x,t$-dependent coefficients appearing in the linearized operator. In particular,
although coefficients involving ratios of the background profile appear
singular near the depleted state, the traveling-wave relations provide the
necessary cancellations. As a result, the coefficients remain bounded and
approach limiting values in the far-field regions.
Using this structure, the linearized mobile-phase equation can be treated as a
variable-coefficient parabolic problem. { 
The maximum-principle argument in Section~\ref{par4} yields a conditional
finite-time a priori bound for the mobile-phase perturbation under prescribed
bounds on the background perturbation and its spatial derivatives. Although
the characteristic representation of the background equation controls
$\bar v$ in terms of $\bar u$, the present argument does not close the
required bounds on $\bar v_s$ and $\bar v_{ss}$. Consequently, the result does
not establish asymptotic stability of the full coupled system; spectral and
nonlinear stability, including treatment of the translational neutral mode,
remain separate problems.
}

Several directions remain for future analysis. The spectral properties of the
variable-coefficient linearized operator can be investigated to determine
whether the traveling bands are stable under localized perturbations. The
nonlinear stability problem requires estimates that control the interaction
between the linearized evolution and the higher-order nonlinear terms.
{  Further investigation of the dependence of the profiles on the transport
parameter $\beta$ may clarify the dynamical behavior associated with the
critical and subcritical regimes.}

From the modeling perspective, the present framework identifies a mechanism
through which localized transport structures can arise from phase coupling
without prescribing an external advection field. The model captures the
interaction between microscopic displacement asymmetry and macroscopic phase
evolution. Although the present work focuses on a one-dimensional setting,
the formulation suggests possible extensions to more general transport
systems with heterogeneous coefficients, additional spatial effects, and
boundary conditions.  {Localized coherent structures have also been reported in multidimensional
reaction--diffusion systems coupled to transport. In excitable media, imposed
linear shear can deform and break propagating excitation waves, while chaotic
advection can generate persistent filamentary structures
\cite{Biktashev1998, Neufeld2001}. More recently, a transition from
one-dimensional localized structures to curved wave segments was observed
experimentally in a delayed optical system and described through an effective
two-dimensional advection--reaction--diffusion model
\cite{MarinoGiacomelli2019}. The mechanism considered here is different:
directed transport is not generated by an externally prescribed velocity
field, but by the relative gradient of the evolving background phase through
the constitutive relation
$E_u=-\chi\,\partial_x\log(v/v_{\rm ref})$. Nevertheless, these
results motivate the question of whether a multidimensional extension of the
present model can support curved or other localized traveling structures.
Establishing the existence and stability of such structures requires a
separate multidimensional analysis.
}

Overall, this work establishes the existence and explicit structure of
traveling-band profiles, develops a reduced scalar description of the profile
equations, and { provides a conditional finite-time a priori estimate for the
mobile-phase perturbation under bounded background forcing}. These results form a basis for further studies of
spectral stability, nonlinear dynamics, and extensions of Einstein-type
nonlinear transport models.

{ 
The explicit traveling-wave profiles also provide quantities against which
the model can be tested once a specific physical interpretation of the two
phases is fixed. In the present formulation, $u$ is the concentration of the
mobile transported phase, whereas identification of $v$ with a particular
measurable geological quantity depends on the application. The diffusion
coefficient $D$ may be estimated independently from transport measurements,
while the effective coefficient $\chi$ may be inferred from the response of
the mobile phase to the measured relative background gradient  $ v_x/v$.

The traveling-wave solution yields several quantitative relations that can
then be compared with observations. In particular, the explicit profiles
depend on the spatial variable through
\[
\frac{c}{D}\cdot(s-s_0),
\]
so that their characteristic spatial length scale is proportional to
\[
\ell=\frac{D}{c}.
\]
Thus, for fixed $D$ and $\beta=\chi/D$, the model predicts that the
characteristic width of the traveling structure decreases inversely with its
propagation speed. The speed $c$ can be obtained from the displacement of the
profile between successive observations, while its spatial width can be
measured directly from the observed profile. Simultaneous measurements of
the two phases provide an additional test through the exact traveling-wave
relation
\[
U=\frac{c}{k}V'.
\]
Furthermore, if the total mobile mass $M$ and the asymptotic background states
are measurable, the independently derived relation
\[
c=\frac{kM}{V_+-V_-}
\]
provides another quantitative consistency condition. These relations identify
specific observational tests of the model, although their application to a
particular geological system requires identification and calibration of the
corresponding physical background variable $v$ and the coefficients
$\chi$ and $k$.
Once application-specific parameter values and physical data are specified,
numerical simulations of the governing system can also be carried out using
standard computational platforms, providing a basis for quantitative
comparison of the model predictions with application-specific observations.
The change in traveling-wave structure across the critical value $\beta=1$
also provides a natural connection with bifurcation questions for nonlinear
traveling waves. In the present analysis, this transition is obtained directly
from the explicit traveling-wave reduction rather than through a
bifurcation-theoretic argument. A systematic bifurcation analysis near
$\beta=1$ is not pursued in the present work.}




\end{document}